\documentclass{amsart}
\usepackage{amsmath,amsthm,amsfonts, amssymb,amscd}
\usepackage[all]{xy}

\newtheorem{theorem}{Theorem}[section]
\newtheorem{lemma}[theorem]{Lemma}%[section]
\newtheorem{definition}[theorem]{Definition}%[section]
\newtheorem{example}[theorem]{Example}%[section]
\newtheorem{remark}[theorem]{Remark}%[section]
\newtheorem{corollary}[theorem]{Corollary}%[section]
\newtheorem{proposition}[theorem]{Proposition}
\newtheorem{problem}[theorem]{Problem}
\newcommand{\minusre}{\hspace{0.3em}\raisebox{0.3ex}{\sl \tiny /}\hspace{0.3em}}
\newcommand{\minusli}{\hspace{0.3em}\raisebox{0.3ex}{\sl \tiny $\setminus $}\hspace{0.3em}}
\newcommand{\lex}{\,\overrightarrow{\times}\,}

\newcommand{\Ker}{\mbox{\rm Ker}}

\newcommand{\Rad}{\mbox{\rm Rad}}

\newcommand{\Infinit}{\mbox{\rm Infinit}}

\newcommand{\ord}{\mbox{\rm ord}}
\begin{document}
\title[Pseudo MV-algebras and Lexicographic Product]{ Pseudo MV-algebras and Lexicographic Product}
\author[Anatolij Dvure\v{c}enskij]{Anatolij Dvure\v{c}enskij$^{1,2}$}
\date{}%
\maketitle
\begin{center}  \footnote{


This work was supported by  the Slovak Research and Development Agency under contract APVV-0178-11,  grant VEGA No. 2/0059/12 SAV, and
CZ.1.07/2.3.00/20.0051.
 }
Mathematical Institute,  Slovak Academy of Sciences,\\
\v Stef\'anikova 49, SK-814 73 Bratislava, Slovakia\\
$^2$ Depart. Algebra  Geom.,  Palack\'{y} University\\
17. listopadu 12, CZ-771 46 Olomouc, Czech Republic\\

E-mail: {\tt dvurecen@mat.savba.sk}
\end{center}

\begin{abstract}
We study algebraic conditions when a pseudo MV-algebra is an interval in the lexicographic product of an Abelian unital $\ell$-group and an $\ell$-group that is not necessary Abelian. We introduce $(H,u)$-perfect pseudo MV-algebras and strong $(H,u)$-perfect pseudo MV-algebras, the latter ones will have a representation by a lexicographic product. Fixing a unital $\ell$-group $(H,u)$, the category of strong $(H,u)$-perfect pseudo MV-algebras is categorically equivalent to the category of $\ell$-groups.
\end{abstract}

\section{Introduction}%1

MV-algebras were introduced by Chang \cite{Cha} as the algebraic counterpart of \L ukasiewicz infinite-valued calculus and during the last 56 years MV-algebras entered deeply in many areas of mathematics and logics. More than 10 years ago, a non-commutative generalization of MV-algebras has been independently appeared. These new algebras are said to be pseudo MV-algebras in \cite{GeIo} or a generalized MV-algebras in \cite{Rac}. For them author \cite{151} generalized a famous Mundici's representation theorem, see e.g. \cite[Cor 7.1.8]{CDM}, showing that every pseudo MV-algebra is always an interval in a unital $\ell$-group not necessarily Abelian. Such algebras have the operation $\oplus$ as a truncated sum and they have two negations. We note that the equality of these two negations does not necessarily imply that a pseudo MV-algebra is an MV-algebra. According to Komori's theorem \cite{Kom}, \cite[Thm 8.4.4]{CDM}, the variety lattice of MV-algebras is countably,  whereas the one of pseudo MV-algebras is uncountable, cf. \cite{Jak, DvHo}. Therefore, the structure of pseudo MV-algebras is much richer than the one of MV-algebras. In \cite{DvHo} it was shown that the class of pseudo MV-algebras where each maximal ideal is normal is a variety. This variety is also very rich and within this variety many important properties of MV-algebras remain.

In \cite{DDT}, perfect pseudo MV-algebras were studied. They are characterized as those that every element of a perfect pseudo MV-algebra is either an infinitesimal or a co-infinitesimal. In \cite{DDT} we have shown that the category of perfect pseudo MV-algebras is equivalent to the variety of $\ell$-groups, and every such an algebra $M$ is in the form of an interval in the lexicographic product $\mathbb Z \lex G$, i.e. $M\cong \Gamma(\mathbb Z \lex G,(1,0))$. This generalized the result from \cite{DiLe1} for perfect MV-algebras. A more general structure, $n$-perfect pseudo MV-algebras were studied in \cite{Dv08}. They can be characterized as those pseudo MV-algebras that have $(n+1)$-comparable slices, and their representation is again in the form of an interval in the lexicographic product $\frac{1}{n}\mathbb Z \lex G$, i.e. every strong $n$-perfect pseudo MV-algebra $M$ is of the form $\Gamma(\frac{1}{n}\mathbb Z \lex G,(1,0))$, where $G$ is any $\ell$-group.
In the paper \cite{264}, we have studied so-called $(\mathbb H,1)$-perfect pseudo MV-algebras, where $(\mathbb H,1)$ is a unital $\ell$-subgroup of the unital $\ell$-group of reals $(\mathbb R,1)$. They can be represented in the form $\Gamma(\mathbb H\lex G,(1,0))$ and such MV-algebras were described in \cite{DiLe2}.

Recently, lexicographic MV-algebras were studied in \cite{DFL}. Such algebras are of the form $\Gamma(H\lex G,(u,0))$, where $(H,u)$ is an Abelian unital $\ell$-group and $G$ is an Abelian $\ell$-group. The main aim of the present paper is to generalize such lexicographic MV-algebras also for the case of pseudo MV-algebras.  Therefore, we introduce so-called $(H,u)$-perfect and strong $(H,u)$-perfect pseudo MV-algebras, where $(H,u)$ is an Abelian unital $\ell$-group. We show that strong $(H,u)$-perfect pseudo MV-algebras are always of the form $\Gamma(H\lex G, (u,0))$, where $G$ is an $\ell$-group not necessarily Abelian. This category will be always categorically equivalent with the variety of $\ell$-groups. Therefore, we generalize many interesting results that were known only in the realm of MV-algebras, see \cite{DiLe2,CiTo, DFL}.

The paper is organized as follows. Section 2 gathers necessary properties of pseudo MV-algebras. Section 3 presents a definition of $(H,u)$-perfect pseudo MV-algebras as those which can be decomposed into a system of comparable slices indexed by the elements of the interval $[0,u]_H=\{h\in H: 0\le h \le u]$, where $(H,u)$ is an Abelian unital $\ell$-group. Section 4 defines strong $(H,u)$-perfect pseudo MV-algebras and we show their representation by $\Gamma(H\lex G,(u,0))$. More details on local pseudo MV-algebras with retractive ideals will be done in Section 5. A free product representation of local pseudo MV-algebras will be done in Section 6. In Section 7 we describe pseudo MV-algebras with a so-called lexicographic ideal. A categorical equivalence of the category of strong $(H,u)$-perfect pseudo MV-algebras will be established in Section 8. Finally, in Section 9 we describe weak $(H,u)$-perfect pseudo MV-algebras as those that they can be represented in the form $\Gamma(H\lex G, (u,g))$, where $g$ is an arbitrary element (not necessarily $g=0$) of an $\ell$-group $G$.

\section{Pseudo MV-algebras}%2

According to \cite{GeIo}, a {\it pseudo MV-algebras} or a GMV-algebra by \cite{Rac} is an algebra $(M;
\oplus,^-,^\sim,0,1)$ of type $(2,1,1,$ $0,0)$ such that the
following axioms hold for all $x,y,z \in M$ with an additional
binary operation $\odot$ defined via $$ y \odot x =(x^- \oplus y^-)
^\sim $$
\begin{enumerate}
\item[{\rm (A1)}]  $x \oplus (y \oplus z) = (x \oplus y) \oplus z;$

\item[{\rm (A2)}] $x\oplus 0 = 0 \oplus x = x;$

\item[{\rm (A3)}] $x \oplus 1 = 1 \oplus x = 1;$

\item[{\rm (A4)}] $1^\sim = 0;$ $1^- = 0;$

\item[{\rm (A5)}] $(x^- \oplus y^-)^\sim = (x^\sim \oplus y^\sim)^-;$

\item[{\rm (A6)}] $x \oplus (x^\sim \odot y) = y \oplus (y^\sim
\odot x) = (x \odot y^-) \oplus y = (y \odot x^-) \oplus
x;$\footnote{$\odot$ has a higher binding priority than $\oplus$.}

\item[{\rm (A7)}] $x \odot (x^- \oplus y) = (x \oplus y^\sim)
\odot y;$

\item[{\rm (A8)}] $(x^-)^\sim= x.$
\end{enumerate}

Any pseudo MV-algebra is a distributive lattice where (A6) and (A7) define the joint $x\vee y$ and the meet  $x\wedge y$ of $x,y$, respectively.

We note that a {\it po-group} (= partially ordered group) is a
group $(G;+,0)$ (written additively) endowed with a partial order $\le$ such that if $a\le b,$ $a,b \in G,$ then $x+a+y \le x+b+y$ for all $x,y \in G.$  We denote by $G^+=\{g \in G: g \ge 0\}$ the {\it positive cone} of $G.$ If, in addition, $G$
is a lattice under $\le$, we call it an $\ell$-group (= lattice
ordered group).  An element $u\in G^+$ is said to be a {\it strong unit}
(= order unit) if $G = \bigcup_n [-nu,nu]$, and the couple $(G,u)$ with a fixed strong unit $u$ is
said to be a {\it unital po-group} or a {\it unital $\ell$-group}, respectively. The {\it commutative center} of a group $H$ is the set $C(H)=\{h\in H: h+h' = h'+h, \ \forall h' \in H\}$. Finally, two unital $\ell$-groups $(G,u)$ and $(H,v)$ is {\it isomorphic} if there is an $\ell$-group isomorphism $\phi:G \to H$ such that $\phi(u)=v$. In a similar way an isomorphism and a homomorphism of unital po-groups are defined.
For more information on po-groups and $\ell$-groups and for unexplained notions about them, see \cite{Fuc, Gla}.

By $\mathbb R$ and $\mathbb Z$ we denote the groups of reals and natural numbers, respectively.

Between pseudo MV-algebras and unital $\ell$-groups there is a very close connection:
If $u$ is a strong unit of a (not necessarily Abelian)
$\ell$-group $G$,
$$
\Gamma(G,u) := [0,u]
$$
and
\begin{eqnarray*}
x \oplus y &:=&
(x+y) \wedge u,\\
x^- &:=& u - x,\\
x^\sim &:=& -x +u,\\
x\odot y&:= &(x-u+y)\vee 0,
\end{eqnarray*}
then $(\Gamma(G,u);\oplus, ^-,^\sim,0,u)$ is a pseudo MV-algebra
\cite{GeIo}.

A pseudo MV-algebra $M$ is an {\it MV-algebra} if $x\oplus y= y\oplus x$ for all $x,y \in M.$ We denote by $\mathcal{P}_s\mathcal{MV}$ and $\mathcal{MV}$ the variety of pseudo MV-algebras and MV-algebras, respectively.

The basic representation theorem for pseudo MV-algebras is the following generalization \cite{151} of the Mundici's famous result:

\begin{theorem}\label{th:2.1}
For any pseudo MV-algebra $(M;\oplus,^-,^\sim,0,1)$,
there exists a unique $($up to iso\-morphism$)$ unital $\ell$-group
$(G,u)$ such that $(M;\oplus,^-,^\sim,0,1)$ is isomorphic to $(\Gamma(G,u);\oplus,^-,^\sim,0,u)$. The
functor $\Gamma$ defines a categorical equivalence of the category
of pseudo MV-algebras with the category of unital $\ell$-groups.
\end{theorem}

We note that the class of pseudo MV-algebras is a variety whereas the class of unital $\ell$-groups is not a variety because it is not closed under infinite products.

Due to this result, if $M=\Gamma(G,u)$ for some unital $\ell$-group $(G,u)$, then $M$ is linearly ordered iff $G$ is a linearly ordered $\ell$-group, see \cite[Thm 5.3]{156}.

Besides a total operation $\oplus$, we can define a partial operation $+$ on any pseudo MV-algebra $M$  in such a way that $x + y$ is defined iff $x \odot y=0$ and then we set
$$x+y := x\oplus y. \eqno(2.1)
$$
In other words, $x+y$ is precisely the group addition $x+y$ if the group sum $x+y$ is defined in $M$.

Let $A,B$ be two subsets of $M$. We define (i) $A\leqslant B$ if $a\le b$ for all $a\in A$ and all $b \in B$, (ii) $A\oplus B=\{a\oplus b: a\in A, b \in B\}$, and (iii) $A + B = \{a+b: $ if $a+b$ exists in $M$ for $a\in A,\ b \in B\}.$ We say that $A+B$ is {\it defined} in $M$ if $a+b$ exists in $M$  for each $a \in A$ and each $b \in B$. (iv) $A^-=\{a^-: a \in A\}$ and $A^\sim = \{a^\sim: a \in A\}.$

Using Theorem \ref{th:2.1}, we have  if $y\le x$, then $x\odot y^-=x-y$ and $y^\sim \odot x = -y+x$, where the subtraction $-$ is in fact the group subtraction in the representing unital $\ell$-group.

We recall that if $H$ and $G$ are two po-groups, then the {\it lexicographic product} $H \lex G$ is the group $H\times G$ which is endowed with the lexicographic order: $(h,g)\le (h_1,g_1)$ iff $h< h_1$ or $h=h_1$ and $g\le g_1$. The lexicographic product $H \lex G$ is an $\ell$-group iff $H$ is linearly ordered group and $G$ is an arbitrary $\ell$-group, \cite[(d) p. 26]{Fuc}. If $u$ is a strong unit for $H$, then $(u,0)$ is a strong unit for $H\lex G$, and $\Gamma(H\lex G,(u,0))$ is a pseudo MV-algebra.

We say that a pseudo MV-algebra $M$ is {\it symmetric} if $x^-=x^\sim$ for all $x \in M$. The pseudo MV-algebra $\Gamma(G,u)$ is symmetric iff $u \in C(G)$, hence the variety of symmetric pseudo MV-algebras is a proper subvariety of the variety $\mathcal {MV}$. For example, $\Gamma(\mathbb R\lex G,(1,0))$ is symmetric and it is an MV-algebra iff $G$ is Abelian.

An {\it ideal} of a pseudo MV-algebra $M$ is any non-empty subset $I$ of $M$ such that (i) $a\le b \in I$ implies $a \in I,$ and (ii) if $a,b \in I,$ then $a\oplus b \in I.$
An ideal is said to be (i) {\it maximal} if $I\ne M$ and it is not a proper subset of another ideal $J \ne M;$ we denote by $\mathcal M(M)$ the set of maximal ideals of $M$, (ii) {\it prime} if $x\wedge y \in I$ implies $x \in I$ or $y \in I$, and (iii) {\it normal} if $x\oplus I=I\oplus x$ for any $x \in M;$ let $\mathcal N(M)$ be the set of normal ideals of $M.$ A pseudo MV-algebra $M$ is {\it local} if there is a unique maximal ideal and, in addition,  this ideal also normal.

There is a one-to-one correspondence between normal ideals and congruences for pseudo MV-algebras, \cite[Thm 3.8]{GeIo}. The quotient pseudo MV-algebra over a normal ideal $I,$  $M/I,$ is defined as the set of all elements of the form $x/I := \{y \in M :
x\odot y^- \oplus y\odot x^- \in I\},$ or equivalently, $x/I := \{y \in M : x^\sim \odot y \oplus y^\sim \odot x \in I\}.$

Let $x\in M$  and an integer $n\ge 0$  be given. We define
$$
  0.x:=0,\quad  1\odot x:=x, \quad (n+1). x := (n. x)\oplus x,
$$
$$
x^0:= 1, \quad x^1 := x,\quad x^{n+1} := x^n \odot x,
$$
$$
0x := 0,\quad 1 x :=x, \quad (n+1)x:= (nx)+ x,
$$
if $nx$ and $(nx)+x$ are defined in $M.$ An element $x$ is said to be an {\it infinitesimal} if $mx$ exists in $M$ for every integer $m \ge 1.$ We denote by $\Infinit(M)$ the set of all infinitesimals of $M.$

We define (i) the {\it radical} of a pseudo MV-algebra $M$,
$\mbox{Rad}(M),$ as the set
$$
\mbox{\rm Rad}(M) = \bigcap\{I:\ I \in {\mathcal M}(M) \},
$$
and (ii) the {\it normal radical} of $M$, $\mbox{Rad}_n(M)$, via
$$
\mbox{Rad}_n(M) = \bigcap\{I:\ I \in {\mathcal N}(M) \cap {\mathcal M}(M)\}
$$
whenever ${\mathcal N}(M) \cap {\mathcal M}(M) \ne \emptyset$.

By \cite[Prop. 4.1, Thm 4.2]{DDJ}, it is possible to show that

$$
\mbox{\rm Rad}(M) \subseteq \mbox{\rm Infinit}(M) \subseteq
\mbox{\rm Rad}_n(M).
$$

The notion of a state is an analogue of a probability measure for pseudo MV-algebras. We say that a mapping $s$ from a pseudo MV-algebra $M$ into the real interval is a {\it state} if (i) $s(a+b)=s(a)+s(b)$ whenever $a+b$ is defined in $M$, and (ii) $s(1)=1$. We define the {\it kernel} of $s$ as the set  $\Ker(s)=\{a \in M: s(a)=0\}$. Then $\Ker(s)$ is a normal ideal of $M$.

If $M$ is an MV-algebra, then at least one state on $M$ is defined. Unlike for MV-algebras, there are pseudo MV-algebras that are stateless, \cite{156} (see also a note just before Theorem \ref{th:4.6} below).  We note that $M$ has at least one state iff $M$ has at least one maximal ideal that is also normal. However, every non-trivial linearly ordered pseudo MV-algebra admits a unique state, \cite[Thm 5.5]{156}.

Let $\mathcal S(M)$ be the set of all states on $M$; it is a convex set. A state $s$ is said to be {\it extremal} if from $s = \lambda s_1 +(1-\lambda) s_2$, where $s_1,s_2 \in \mathcal S(M)$ and $0<\lambda<1$, we conclude $s=s_1=s_2.$ Let $\partial_e\mathcal S(M)$ denote the set of extremal states. In addition, in view of \cite{156}, a state $s$ is extremal iff $\Ker(s)$ is a maximal ideal of $M$ iff $s(a\wedge b)=\min\{s(a),s(b)\}$. Or equivalently, $s$ is a {\it state morphism}, i.e., $s$ is a homomorphism from $M$ into the MV-algebra $\Gamma(\mathbb R,1)$. In addition, a normal ideal $I$ is maximal iff $I=\Ker(s)$ for some extremal state $s$.

We say that a net of states $\{s_\alpha\}_\alpha$ {\it converges weakly} to a state $s$ if $s(a)= \lim_\alpha s_\alpha(a)$, $a \in M$. Then $\mathcal S(M)$ and $\partial_e\mathcal S(M)$ are compact Hausdorff topological spaces, in particular cases both can be empty, and due to the Krein-Mil'man Theorem \cite[Thm 5.17]{Goo}, every state on $M$ is a weak limit of a net of convex combinations of extremal states.

Pseudo MV-algebras can be studied also in the frames of pseudo effect algebra which are a non-commutative generalization of effect algebras introduced by \cite{FoBe}.

According to \cite{DvVe1, DvVe2}, a partial algebraic structure
$(E;+,0,1),$  where $+$ is a partial binary operation and 0 and 1 are constants, is called a {\it pseudo effect algebra} if, for all $a,b,c \in E,$ the following hold:
\begin{itemize}
\item[{\rm (PE1)}] $ a+ b$ and $(a+ b)+ c $ exist if and only if $b+ c$ and $a+( b+ c) $ exist, and in this case,
$(a+ b)+ c =a +( b+ c)$;

\item[{\rm (PE2)}] there are exactly one  $d\in E $ and exactly one $e\in E$ such
that $a+ d=e + a=1$;

\item[{\rm (PE3)}] if $ a+ b$ exists, there are elements $d, e\in E$ such that
$a+ b=d+ a=b+ e$;

\item[{\rm (PE4)}] if $ a+ 1$ or $ 1+ a$ exists,  then $a=0.$
\end{itemize}

If we define $a \le b$ if and only if there exists an element $c\in
E$ such that $a+c =b,$ then $\le$ is a partial ordering on $E$ such
that $0 \le a \le 1$ for any $a \in E.$ It is possible to show that
$a \le b$ if and only if $b = a+c = d+a$ for some $c,d \in E$. We
write $c = a \minusre b$ and $d = b \minusli a.$ Then

$$ (b \minusli a) + a = a + (a \minusre b) = b,
$$
and we write $a^- = 1 \minusli a$ and $a^\sim = a\minusre 1$ for any
$a \in E.$

If $(G,u)$ is a unital po-group, then $(\Gamma(G,u);+,0,u),$ where
the set $\Gamma(G,u):=\{g\in G: 0\le g \le u\}$ is endowed with the restriction of the group addition $+$ to $\Gamma(G,u)$ and with $0$ and $u$ as $0$ and $1$, is a pseudo effect algebra. Due to \cite{DvVe1, DvVe2}, if a pseudo effect algebra satisfies a special type of the Riesz Decomposition Property, RDP$_1$, then every pseudo effect algebra is an interval in some unique (up to isomorphism of unital po-groups) $(G,u)$ satisfying also RDP$_1$ such that $M \cong \Gamma(G,u)$.

We say that a mapping $f$ from one pseudo effect algebra $E$ onto a second one $F$ is a {\it homomorphism} if (i) $a,b\in E$ such that $a+b$ is defined in $E$, then $f(a)+f(b)$ is defined in $F$ and $f(a+b)=f(a)+f(b)$, and (ii) $f(1)=1$.

We say that a pseudo effect algebra $E$ satisfies RDP$_2$ property if $a_1+a_2=b_1+b_2$ implies that there are four elements $c_{11},c_{12},c_{21},c_{22}\in E$ such that (i) $a_1 = c_{11}+c_{12},$ $a_2= c_{21}+c_{22},$ $b_1= c_{11} + c_{21}$ and $b_2= c_{12}+c_{22}$, and (ii) $c_{12}\wedge c_{21}=0.$

In \cite[Thm 8.3, 8.4]{DvVe2}, it was proved that if $(M; \oplus,^-,^\sim,0,1)$ is a  pseudo MV-algebra, then $(M;+,0,1),$ where $+$ is defined by (2.1), is a pseudo effect algebra with RDP$_2.$  Conversely, if $(E; +,0,1)$ is a pseudo effect algebra with RDP$_2,$ then $E$ is a lattice, and by \cite[Thm 8.8]{DvVe2}, $(E; \oplus,^-,^\sim,0,1),$ where
$$
a\oplus b := (b^-\minusli (a\wedge b^-))^\sim,\eqno(2.2)
$$
is a pseudo MV-algebra. In addition, a pseudo effect algebra $E$ has RDP$_2$ iff $E$ is a lattice and $E$ satisfies RDP$_1,$ see \cite[Thm 8.8]{DvVe2}.

Finally, we note that an {\it ideal} of a pseudo effect algebra $E$ is a non-empty subset $I$ such that (i) if $x,y\in I$ and $x+y$ is defined in $E$, then $x+y \in I$, and (ii) $x\le y \in I$ implies $x\in I$. An ideal $I$ is {\it normal} if $a+I:=\{a+i: i \in I$ if $a+i$ exists in $E\}=I+a:=\{j+a: j \in I\}$ for any $a\in E$. A maximal ideal is defined in a standard way. If $M$ is a pseudo MV-algebra, then the ideal $I$ of $M$ is also an ideal when $M$ is understood as a pseudo effect algebra; this follows from the fact $x\oplus y = (x\wedge y^-)+y$.

We note that a mapping from a pseudo effect algebra $E$ into a pseudo effect algebra $F$ is a {\it homomorphism} if (i) $a+b\in E$ implies $h(a)+h(b)\in F$ and $h(a+b)=h(a)+h(b)$, and (ii) $h(1)=1$. A bijective mapping $h:E\to F$ is an {\it isomorphism} if both $h$ and $h^{-1}$ are homomorphisms of pseudo effect algebras.

\section{$(H,u)$-Perfect Pseudo MV-algebras}%3

Generalizing ideas from \cite{DiLe1, DDT, Dv08, 264}, we introduce the basic notions of our paper.

If $(H,u)$ is a unital $\ell$-group, we set $[0,u]_H:=\{h\in H: 0\le h\le u\}.$

\begin{definition}\label{de:3.1}
{\rm Let $(H,u)$ be a linearly ordered group and let $u$ belong to the commutative center $C(H)$ of $H$.
We say that a pseudo MV-algebra $M$ is $(H,u)$-{\it perfect}, if there is a system $(M_t: t \in [0,u]_H)$ of nonempty subsets of $M$ such that it is an $(H,u)$-{\it decomposition} of $M,$ i.e. $M_s \cap M_t= \emptyset$ for $s<t,$ $s,t \in [0,u]_ H$ and $\bigcup_{t\in [0,u]_ H}M_t = M$, and \begin{enumerate}

\item[{\rm (a)}]
$M_s \leqslant M_t$ for all $ s<t$, $s,t \in [0,u]_ H$;

\item[{\rm (b)}]
$M_t^- = M_{u-t} =M_t^\sim$ for each $t \in [0,u]_H$;

\item[{\rm (c)}]  if $x \in M_v$ and $y \in M_t$, then
$x\oplus y \in M_{v\oplus t},$ where $v\oplus t=\min\{v+t,u\}.$

\end{enumerate}}
\end{definition}

We note that  (a) can be written equivalently in a stronger way:  if $s<t$ and $a \in M_s$ and $b \in M_t$, then $a<b$. Indeed, by (b) we have $a\le b$. If $a=b$, then $a\in M_s\cap M_t=\emptyset,$ which is absurd. Hence, $a<b$.

In addition, (i) if $(H,u)=(\mathbb Z,1)$ and $M$ is an MV-algebra, we are speaking on  a {\it perfect} MV-algebra, \cite{DiLe1}, (ii)  if $(\mathbb H,u)=(\frac{1}{n}\mathbb Z,1),$ a $(\frac{1}{n}\mathbb Z,1)$-perfect pseudo MV-algebra is said to be $n$-{\it perfect}, see \cite{Dv08}, (iii) if $\mathbb H$ is a subgroup of the group of real numbers $\mathbb R$, such that $1 \in \mathbb H$, $(\mathbb H,1)$-perfect pseudo MV-algebras are in \cite{264} called $\mathbb H$-{\it perfect} pseudo MV-algebras.

For example, let
$$
M =\Gamma(H\lex G, (u,0)), \eqno(3.1)
$$
where $u \in C(H)$. We set $M_0=\{(0,g): g \in G^+\},$ $M_u:=\{(u,-g): g \in G^+\}$ and
for  $t \in [0,u]_ H\setminus \{0,u\},$ we define $M_t:=\{(t,g): g \in G\}.$ Then $(M_t: t \in [0,u]_H)$ is an $(H,u)$-decomposition of $M$ and $M$ is an $(H,u)$-perfect pseudo MV-algebra.

%Sometimes we will write also $M=(M_t: t \in [0,u]_ H)$ for an $(H,u)$-perfect pseudo MV-algebra.

As a matter of interest, if $O$ is the zero group, then $\Gamma(O\lex G, (0,0))$ is a one-element pseudo MV-algebra. The pseudo MV-algebra $\Gamma(\mathbb Z\lex O,(1,0))$ is a two-element Boolean algebra. In general, $\Gamma(H\lex O,(u,0))\cong \Gamma(H,u)$ and it is semisimple (that is, its radical is a singleton) iff $H$ is Archimedean. If $G \ne O\ne H$, $\Gamma(H\lex G,(u,0))$ is not semisimple.

\begin{theorem}\label{th:3.2}
Let $M =(M_t: t \in [0,u]_ H)$ be an  $(H,u)$-perfect pseudo MV-algebra.

\begin{enumerate}

\item[{\rm (i)}] Let $a \in M_v,$  $b \in M_t$. If $v+t < u$, then
$ a+b $ is defined in $M$ and $a+b \in M_{v+t}$; if $a+b$ is defined
in $M$, then $v+t\le u$. If $a+b$ is defined in $M$ and $v+t=u$, then $a+b \in M_u$.

\item[{\rm (ii)}]
$M_v + M_t$ is defined in $M$ and $M_v + M_t = M_{v+t}$ whenever $v+t < u$.

\item[{\rm (iii)}] If $a \in M_v$ and $b \in M_t$, and
$v+t > u$, then $a+b$ is not defined in $M.$

\item[{\rm (iv)}] If $a \in M_v$ and $b \in M_t$, then $a\vee b \in M_{v\vee t}$ and $a\wedge b \in M_{v\wedge t}.$

\item[{\rm (v)}]
$M$ admits a state $s$ such that  $M_0 \subseteq \Ker(s)$.

\item[{\rm (vi)}]  $M_0$ is a normal ideal
of $M$ such that $M_0 + M_0 = M_0$ and $M_0\subseteq \Infinit(M).$

\item[{\rm (vii)}] The quotient pseudo MV-algebra $M/M_0\cong \Gamma(H,u).$

\item[{\rm (viii)}]
Let $M = (M'_t: t \in [0,u]_ H)$
be another $(H,u)$-decomposition of $M$ satisfying {\rm (a)--(c)} of Definition {\rm \ref{de:3.1}}, then
$M_t = M_t'$ for each $t \in [0,u]_ H.$

\item[{\rm (ix)}] $M_0$ is a prime ideal of $M$.

\end{enumerate}
\end{theorem}

\begin{proof}
(i)  Assume $a\in M_v$ and $b \in M_t$. If
$v+t< u$, then $b^- \in M_{u-t}$, so that $a \le b^-$, and $a+b$ is
defined in $M.$  Conversely, let $a+b$ be defined, then $a\le b^-\in
M_{u-t}$. If $v+t>u$, then $v>u-t$ and $a>b^-\ge a$ which is absurd, and
this gives $v+t\le u$.  Now let $v+t=u$ and $a+b$ be defined in $M$. By (c) of Definition \ref{de:3.1}, we have $a+b \in M_u$.

(ii) By (i), we have  $M_v +M_t \subseteq M_{v+t}$.
Suppose $z \in M_{v+t}$.  Then, for any $x\in M_v,$ we have $x \le z$.
Hence, $y = z- x$ is defined in $M$ and $y \in M_w$ for
some $w \in [0,u]_ H.$ Since $z = y+x \in M_{v+t}\cap M_{v+w}$,  we conclude $t=w$ and $M_{v+t} \subseteq M_v + M_t.$

(iii) It follows from (i).

(iv) Inasmuch as $x\wedge y = (x\oplus y^\sim)-y^\sim$, we have by (c) of Definition \ref{de:3.1}, $(x\oplus y^\sim)-y^\sim \in M_s$, where $s =((v+u-t)\wedge u) -(u-t)= v\wedge t.$ Using a de Morgan law and property (d), we have $x\vee y \in M_{v \vee t}.$

(v)  Let $s_0$ be a unique state on $\Gamma(H,u)$ which exists in view of \cite[Thm 5.5]{156}. Define a mapping $s:\ M \to [0,1]$ by $s(x) = s_0(t)$ if $x \in M_t$.  It is clear that $s$ is a well-defined mapping.  Take $a,b \in M$ such
that $a+ b$ is defined in $M$.  Then there are unique indices $v$ and $t$ such that $a\in M_v$ and $b \in M_t$. By (i), $v+t \le u$
and  $a+b \in M_{v+t}$. Therefore, $s(a+b) = s_0(v+t)  = s_0(v)+s_0(t)= s(a) + s(b).$ It is evident that $s(1) = 1$ and   $M_0\subseteq \Ker(s)$.

(vi) It is clear that $M_0$ is an ideal of $M$. To prove the normality of $M_0$, take $x\in M_v$ and $y\in M_0$.  Then $x\oplus y = ((x\oplus y)-x)+x\in M_v$ which implies by (i)--(ii) $(x\oplus y)-x \in M_0$ and $x\oplus M_0\subseteq M_0\oplus x$. In the same way we proceed for the second implication.

Since $M_0+M_0=M_0,$ by (ii) we have $M_0 \subseteq \Infinit(M).$

(vii)  Since by (vi) $M_0$ is a normal ideal, $M/M_0$ is a pseudo MV-algebra, too. Using (iv), it is easy to verify that $x\sim_{M_0} y$ iff there is an $h \in [0,u]_H$ such that $x,y \in M_h$. We define a mapping $\phi:M/M_0\to \Gamma(H,u)$ by $\phi(x)=h$ iff $x \in M_h$ for some $h \in [0,u]_H.$ The mapping $\phi$ is an isomorphism from $M/M_0$ onto $\Gamma(H,u)$.

(viii)  Let $M = (M_t': t \in [0,u]_ H)$ be another $(H,u)$-decomposition of $M$. We assert  $M_0=M_0'$. If not, there are $x\in M_0\setminus M_0'$ and $y\in M_0' \setminus M_0$.  By (a), we have $x<y$ as well as $y<x$ which is absurd. Hence, $M_0=M_0'$. By (vi), $M_0$ is normal and by (vii), $M_0\cong \Gamma(H,u)\cong M/M_0'$. If $x\sim_{M_0} y$, then $x,y \in M_h$ for some $h \in [0,u]_H$, as well as $x\sim_{M_0'} y$ implies $x,y \in M_{t'}$ and $t=t'$ which implies $M_t = M_t'$ for any $t \in [0,1]_H$.

(ix) By (vii), $M/M_0\cong \Gamma(H,u)$, so that $M/M_0$ is a linear pseudo MV-algebra. Applying \cite[Thm 6.1]{156}, we conclude that the normal ideal $M_0$ is prime.
\end{proof}

\begin{example}\label{ex:3.3}
We define MV-algebras: $M_1=\Gamma(\mathbb Z \lex (\mathbb Z\lex \mathbb Z),(1,(0,0)))$, $M_2 = \Gamma((\mathbb Z \lex \mathbb Z)\lex \mathbb Z,((1,0),0))$, and $M=\Gamma(\mathbb Z \lex \mathbb Z\lex \mathbb Z,(1,0,0))$ which are  mutually isomorphic. The first one is $(\mathbb Z,1)$-perfect, the second one is $(\mathbb Z\lex \mathbb Z,(1,0))$-perfect and of course, the linear unital $\ell$-groups $(\mathbb Z,1)$ and $(\mathbb Z \lex \mathbb Z,(1,0))$ are not isomorphic while the first one is Archimedean and the second one is not Archimedean. Both pseudo MV-algebras define the corresponding natural $(\mathbb Z,1)$-decomposition $(M^1_t)_t$ and $(\mathbb Z \lex \mathbb Z,(1,0))$-decompositions $(M^2_q)_q$ of $M_1$ and $M_2$, respectively. Then $M^1_0=\{(0,(0,m)): m \ge 0\}\cup \{(0,(n,m)): n>0, m\in \mathbb Z\}=\Ker(s_1)=\Infinit(M_1)$, where $s_1$ is a unique state on $M_1$; it is two-valued. But $M^2_0=\{((0,0),m): m\ge0\}  \subset \Ker(s_2)= \Infinit(M_2)$, where $s_2$ is a unique state on $M_2$, it vanishes only on $\Infinit(M_2)$; it is two-valued.
\end{example}

In what follows, we show that the normal ideal $M_0$ of an $(H,u)$-decomposition $(M_t: t \in [0,u]_H)$ of an $(H,u)$-perfect pseudo MV-algebra is maximal iff $(H,u)$ is isomorphic with $(\mathbb H,1)$, where $\mathbb H$ is a subgroup of the group of reals $\mathbb R$ with natural order, and $1 \in \mathbb H$.

\begin{theorem}\label{th:3.4}
Let $(M_t: t \in [0,u]_H)$ be an $(H,u)$-decomposition of an $(H,u)$-perfect pseudo MV-algebra $M$.  The following statements are equivalent:
\begin{enumerate}
\item[{\rm (i)}] $M_0$ is maximal.
\item[{\rm (ii)}] $(H,u)$ is isomorphic to some $(\mathbb H,1)$, where $\mathbb H$ is a subgroup of the group $\mathbb R$ and $1 \in \mathbb H$.
\item[{\rm (iii)}] $M$ possesses a unique state $s$ and $M_0=\Ker(s)$.
\end{enumerate}
\end{theorem}

\begin{proof}
(i) $\Rightarrow$ (ii). By \cite[Prop 3.4-3.5]{156}, $M_0$ is maximal iff $M/M_0$ is simple, i.e. it does not contain any  non-trivial proper ideal. By (vii) of Theorem \ref{th:3.2}, $(M/M_0,u/M_0)\cong \Gamma(H,u)$ which means by Theorem \ref{th:2.1} that $(H,u)$ is a linear, Archimedean and Abelian unital $\ell$-group, and by H\"older's theorem, \cite[Thm XIII.12]{Bir} or \cite[Thm IV.1.1]{Fuc}, it is isomorphic to some $(\mathbb H,1)$, where $\mathbb H$ is a subgroup of $\mathbb R$ and $1 \in \mathbb H$.

(ii) $\Rightarrow$ (iii). If $(H,u)\cong (\mathbb H,1)$, where $\mathbb H$ is a subgroup of $\mathbb R$ and $1 \in \mathbb H$, then $M$ is isomorphic to an $(\mathbb H,1)$-perfect pseudo MV-algebra. This kind of pseudo MV-algebras was studied in \cite{264}, and by \cite[Thm 3.2(iv)]{264}, $M$ possesses a unique state $s$. This state has the property $s(M)=\mathbb H$ and $\Ker (s)=M_0$.

(iii) $\Rightarrow$ (i). If $s$ is a unique state of $M$ and $M_0=\Ker(s)$, by \cite{156}, $M_0$ is a normal and maximal ideal of $M$.
\end{proof}

\begin{remark}\label{re:3.5}
{\rm We note that in Corollary \ref{co:state2} it will be shown that if an $(H,u)$-perfect pseudo MV-algebra $M$ is of a stronger form, namely a strong $(H,u)$-perfect pseudo MV-algebra introduced in the next section, then $M$ has a unique state. In general, the uniqueness of a state for any $(H,u)$-perfect pseudo MV-algebra is unknown.}
\end{remark}

%\begin{remark}\label{re:3.5}
{\rm %If a pseudo MV-algebra $M$ is $(\mathbb H,1)$-perfect for some subgroup $\mathbb H$ of $\mathbb R$, then $M$ has a unique $(\mathbb H,1)$-decomposition, see \cite[Thm 3.2(vii)]{264}. We do not know whether this is true also for arbitrary $(H,u)$-perfect pseudo MV-algebra $M$.  YES

%Does have $M$ a unique state ? If yes, then it is not necessarily such that $\Ker(s)=M_0$, see the pseudo MV-algebra $M_2$ from Example \ref{ex:3.3}. ??? }
%\end{remark}

We note that there is uncountably many non-isomorphic unital  $\ell$-subgroups $(\mathbb H,1)$ of the unital group $(\mathbb R,1)$. By \cite[Lem 4.21]{Goo}, every $\mathbb H$ is either {\it cyclic}, i.e. $\mathbb H=\frac{1}{n}\mathbb Z$ for some $n \ge 1$ or $\mathbb H$ is dense in $\mathbb R.$

Therefore, if $\mathbb H= \mathbb H(\alpha)$ is a subgroup of $\mathbb R$ generated by $\alpha \in [0,1]$ and $1,$ then $\mathbb H = \frac{1}{n}\mathbb Z$ for some integer $n\ge 1$ if $\alpha$ is a rational number. Otherwise, $\mathbb H(\alpha)$ is countable and dense in $\mathbb R,$ and $M(\alpha):= \Gamma(\mathbb H(\alpha),1)=\{m+n\alpha: m,n \in \mathbb Z,\ 0\le m+n\alpha \le 1\},$ see \cite[p. 149]{CDM}. Therefore, we have uncountably many non-isomorphic $(\mathbb H,1)$-perfect pseudo MV-algebras.

\section{Representation of Strong $(H,u)$-perfect Pseudo MV-algebras}%4

In accordance with \cite{264}, we introduce the following notions and generalize results from \cite{264} for strong $(H,u)$-perfect pseudo MV-algebras. Our aim is to find an algebraic characterization of pseudo MV-algebras that can be represented in the form of the lexicographic product
$$
\Gamma(H\lex G, (u,0)),
$$
where $(H,u)$ is a linearly ordered Abelian unital $\ell$-group and $G$ is an $\ell$-group not necessarily Abelian. In \cite{264}, we have studied a particular case when $(H,u)= (\mathbb H,1)$, where $\mathbb H$ is a subgroup of reals.

%We introduce a stronger notion of $\mathbb H$-perfect pseudo MV-algebras, called strong $\mathbb H$-perfect pseudo MV-algebras, and we show when it can be represented in the form $\Gamma(\mathbb H\lex G,(1,0))$ for some unital $\ell$-group $G.$

We say that a pseudo MV-algebra $M$ enjoys
{\it unique  extraction of roots of $1$} if $a,b \in M$ and $n a, nb$ exist in $M$, and $n a=1= n b$, then $a= b.$ Every linearly ordered pseudo MV-algebra enjoys due to Theorem \ref{th:2.1} and \cite[Lem 2.1.4]{Gla}  unique extraction of roots. In addition, every pseudo MV-algebra $\Gamma(H \lex G,(u,0))$, where $(H,u)$ is a linearly ordered $\ell$-group, enjoys unique  extraction of roots of $1$ for any $n\ge 1$
and for any $\ell$-group $G$.  Indeed, let $k(s,g) = (u,0)=k(t,h)$ for some $s,t \in [0,u]_ H,$ $g,h \in G$, $k\ge 1.$ Then $ks=u=kt$ which yields $s=t >0$, and $kg=0=kh$ implies $g=0=h.$

Let $n\ge 1$ be an integer. An element $a$ of a pseudo MV-algebra $M$ is said to be {\it cyclic of order} $n$ or simply {\it cyclic} if $na$ exists in $M$ and $na =1.$ If $a$ is a cyclic element of order $n$, then $a^- = a^\sim$, indeed, $a^- = (n-1) a = a^\sim$. It is clear that $1$ is a cyclic element of order $1.$

Let $M=\Gamma(G,u)$ for some unital $\ell$-group $(G,u).$ An element
$c\in M$ such that (a) $nc=u$ for some integer $n \ge 1,$ and (b) $c\in   C(G),$ where $C(G)$ is a commutative center of $G,$
is said to be a {\it strong cyclic element of order $n$.}

We note that if  $\mathbb H$ is a subgroup of reals and $t=1/n,$ then $c_\frac{1}{n}$ is a strong cyclic element of order $n.$

For example, the pseudo MV-algebra $M:=\Gamma(\mathbb Q \lex G, (1,0)),$ where $\mathbb Q$ is the group of rational numbers, for every integer $n\ge 1,$ $M$ has a unique cyclic element of order $n,$ namely $a_n =(\frac{1}{n},0).$ The pseudo MV-algebra $\Gamma(\frac{1}{n}\mathbb Z, (1,0))$ for a prime number $n\ge 1,$ has  the only cyclic element of order $n,$ namely $(\frac{1}{n},0).$ If $M=\Gamma(G,u)$ and $G$ is a representable $\ell$-group, $G$ enjoys unique extraction of roots of $1,$ therefore, $M$ has at most one cyclic element of order $n.$ In general, a pseudo MV-algebra $M$ can have two different cyclic elements of the same order. But if $M$ has a strong cyclic element of order $n,$ then it has a unique strong cyclic element of order $n$ and a unique cyclic element of order $n,$ \cite[Lem 5.2]{DvKo}.

We say that an $(H,u)$-decomposition $(M_t: t\in [0,u]_H)$ of $M$ has the {\it cyclic property} if there is a system of elements $(c_t\in M: t \in [0,u]_ H)$ such that (i) $c_t \in M_t$ for any $t \in [0,u]_ H,$ (ii) if $v+t \le 1,$ $v,t \in [0,u]_ H,$ then $c_v+c_t=c_{v+t},$ and (iii) $c_1=1.$ Properties: (a) $c_0=0;$ indeed, by (ii) we have $c_0+c_0=c_0,$ so that $c_0=0.$ (b) If $t=1/n,$ then $c_\frac{1}{n}$ is a cyclic element of order $n.$

Let $M =\Gamma(G,u),$ where $(G,u)$ is a unital $\ell$-group. An $(H,u)$-decomposition $(M_t: t\in [0,u]_ H)$ of $M$ has the {\it strong cyclic property} if there is a system of elements $(c_t\in M: t \in [0,u]_ H)$, called an $(H,u)$-{\it strong cyclic family}, such that

\begin{enumerate}
\item[{\rm (i)}] $c_t \in M_t\cap C(G)$ for each $t \in [0,u]_ H$;
\item[{\rm (ii)}] if $v+t \le 1,$ $v,t \in [0,u]_ H,$ then $c_v+c_t=c_{v+t}$;
\item[{\rm (iii)}] $c_1=1.$
\end{enumerate}

For example, let $M=\Gamma(H \lex G,(u,0)),$ where $(H,u)$ is an Abelian linearly ordered unital $\ell$-group and $G$ is an $\ell$-group (not necessarily Abelian), and $M_t =\{(t,g): (t,g)\in M\}$ for $t \in [0,u]_ H.$ If we set $c_t =(t,0),$ $t \in [0,u]_ H,$ then the system $(c_t: t \in [0,u]_ H)$ satisfies (i)---(iii) of the strong cyclic property, and $(M_t: t \in [0,u]_ H)$ is an $(H,u)$-decomposition of $M$ with the strong cyclic property.

Finally, we say that a pseudo MV-algebra $M$ is {\it strong} $(H,u)$-{\it perfect} if there is an  $(H,u)$-decomposition $(M_t: t \in [0,u]_ H)$ of $M$ with the strong cyclic property.

A prototypical example of a strong $(H,u)$-perfect pseudo MV-algebra is the following.

\begin{proposition}\label{pr:3.4}
Let $G$ be an  $\ell$-group and $(H,u)$ an Abelian unital $\ell$-group. Then  the pseudo MV-algebra
$$
\mathcal M_{H,u}(G):=\Gamma(H \lex G,(u,0)) \eqno(4.1)
$$
is a strong $(H,u)$-perfect pseudo MV-algebra with a strong cyclic family $((h,0): h \in [0,u]_H)$.
\end{proposition}

Now we present a representation theorem for strong $(H,u)$-perfect pseudo MV-algebras by (4.1). The following theorem uses the basic ideas of the particular situation $(H,u)= (\mathbb H,1)$ which was proved in \cite[Thm 4.3]{264}.

\begin{theorem}\label{th:3.5}
Let $M$ be a strong $(H,u)$-perfect pseudo MV-algebra, where $(H,u)$ is an Abelian unital linearly ordered $\ell$-group. Then there is a unique (up to isomorphism) $\ell$-group $G$  such that $M \cong \Gamma(H \lex G,(u,0)).$
\end{theorem}

\begin{proof}
Since $M$ is a pseudo MV-algebra, due to \cite[Thm 3.9]{151}, there is a unique unital (up to isomorphism of unital $\ell$-groups)  $\ell$-group $(K,v)$ such that $M \cong \Gamma(K,v).$ Without loss of generality we can assume that $M=\Gamma(K,v)$. Assume $(M_t: t \in [0,u]_H)$ is an $(H,u)$-decomposition of $M$ with the strong cyclic property and with an $(H,u)$-strong cyclic family $(c_t \in M: t \in [0,u]_ H)$.

By (vi) of Theorem \ref{th:3.2}, $M_0$ is an associative cancellative semigroup satisfying conditions of Birkhoff's Theorem  \cite[Thm XIV.2.1]{Bir}, \cite[Thm II.4]{Fuc},
which guarantees that $M_0$ is a positive cone of
a unique (up to isomorphism) directed po-group $G$. Since $M_0$ is a lattice, we have that $G$ is an $\ell$-group.

Take the $(H,u)$-strong perfect pseudo MV-algebra $\mathcal M_{H,u}(G)$ defined by (4.1), and define a mapping $\phi: M \to \mathcal M_{H,u}(G)$ by

$$
\phi(x):= (t, x - c_t)\eqno (4.2)
$$
whenever $x \in M_t$ for some $t \in [0,u]_ H,$ where $ x-c_t$ denotes the difference taken in the group $K$.

\vspace{2mm}
{\it Claim 1:} {\it  $\phi$ is a well-defined mapping.}
 \vspace{2mm}

Indeed, $M_0$  is in fact the positive cone of an $\ell$-group $G$ which is a subgroup of $K.$    Let $x \in M_t.$ For the element $x - c_t \in K,$ we define $(x-c_t)^+:= (x-c_t)\vee 0 = (x \vee c_t)-c_t \in M_0$ (when we use (iii) of Theorem \ref{th:3.2}) and similarly $(x -c_t)^- := -((x-c_t)\wedge 0) = c_t - (x\wedge c_t) \in M_0.$ This implies that $x-c_t= (x-c_t)^+ - (x-c_t)^-\in G.$

\vspace{2mm}
{\it Claim 2:} {\it The mapping $\phi$ is an injective and surjective homomorphism of pseudo effect algebras.}

\vspace{2mm}

We have $\phi(0)=(0,0)$ and $\phi(1)=(1,0).$ Let $x \in M_t.$ Then $x^- \in M_{1-t},$ and $\phi(x^-) =(1-t, x -  c_{1-t}) = (1,0)-(t,x - c_t)=\phi(x)^-.$ In an analogous way, $\phi(x^\sim)=\phi(x)^\sim.$

Now let $x,y \in M$ and let $x+y$ be defined in $M.$ Then $x\in M_{t_1}$ and $y \in M_{t_2}.$ Since $x\le y^-,$ we have $t_1 \le 1-t_2$ so that $\phi(x) \le \phi(y^-)=\phi(y)^-$ which means  $\phi(x)+\phi(y)$ is defined in $\mathcal M_{H,u} (G).$ Then $\phi(x+y) = (t_1+t_2, x+y - c_{t_1+t_2}) =
(t_1+t_2, x+y -(c_{t_1} + c_{t_2}))= (t_1,x-c_{t_1}) + (t_2,y- c_{t_2})=\phi(x)+\phi(y).$

Assume $\phi(x)\le \phi(y)$ for some $x\in M_{t}$ and $y \in M_v.$ Then $(t,x-c_t)\le (v, y - c_v).$ If $t=v,$ then $x-c_t\le y-c_t$ so that $x\le y.$  If $i<j,$ then $x \in M_t$ and $y\in M_v$ so that $x<y.$  Therefore, $\phi$ is injective.

To prove that $\phi$ is surjective, assume two cases: (i) Take $g \in G^+=M_0.$  Then $\phi(g)=(0,g).$ In addition $g^- \in M_1$ so that $\phi(g^-) = \phi(g)^-= (0,g)^- = (1,0)-(0,g)=(1,-g).$ (ii) Let $g \in G$ and $t$ with $0<t<1$ be given. Then $g = g_1-g_2,$ where $g_1,g_2 \in G^+=M_0.$ Since $c_t \in M_t,$ $g_1 + c_t$ exists in $M$ and it belongs to $M_t,$ and $g_2 \le g_1+c_t$ which yields $(g_1+c_t)- g_2 = (g_1+c_t)\minusli g_2 \in M_t.$  Hence, $g+c_t = (g_1 + c_t)\minusli g_2 \in M_t$ which entails $\phi(g+c_t)=(t,g).$

\vspace{2mm}
{\it Claim 3:}  {\it If $x\le y,$ then $\phi(y \minusli x)=\phi(y)\minusli \phi(x)$ and $\phi(x\minusre y)=\phi(x)\minusre \phi(y).   $}
\vspace{2mm}

It follows from the fact that $\phi$ is a homomorphism of pseudo effect algebras.

\vspace{2mm}
{\it Claim 4:} $\phi(x\wedge y) = \phi(x)\wedge \phi(y)$ and $\phi(x\vee y)=\phi(x)\vee\phi(y).$
\vspace{2mm}

We have, $\phi(x), \phi(y) \ge \phi(x\wedge y).$ If $\phi(x), \phi(y) \ge \phi(w)$ for some $w \in M,$ we have $x,y \ge w$ and $x\wedge y \ge w.$ In the same way we deal with $\vee.$

\vspace{2mm}
{\it Claim 5:} {\it $\phi$ is a homomorphism of pseudo MV-algebras.}
\vspace{2mm}

It is necessary to show that $\phi(x\oplus y)=\phi(x)\oplus \phi(y).$
This follows straightforward from the previous claims and equality (2.2).

Consequently, $M$ is isomorphic to $\mathcal M_{H,u}(G)$ as pseudo MV-algebras.

If $M \cong \Gamma(H\lex G',(u,0))$ for some $G'$, then $(H\lex G,(u,0))$ and $(H\lex G',(u,0))$ are isomorphic unital $\ell$-groups in view of the categorical equivalence, see \cite[Thm 6.4]{151} or Theorem \ref{th:2.1}; let $\psi: \Gamma(H\lex G,(u,0)) \to \Gamma(H\lex G',(u,0))$ be an isomorphism of the lexicographic products. Hence, by Theorem \ref{th:3.2}(viii), we see that $\psi(\{(0,g): g \in G^+\})= \{(0,g'): g' \in G'^+\}$ which proves that $G$ and $G'$ are isomorphic $\ell$-groups.
\end{proof}

We say that a pseudo MV-algebra is {\it lexicographic} if there are an Abelian linearly ordered unital $\ell$-group $(H,u)$ and an $\ell$-group $G$ (not necessarily Abelian) such that $M \cong \Gamma(H\lex G,(u,0))$. In other words, by Theorem \ref{th:3.5}, $M$ is lexicographic iff $M$ is strong $(H,u)$-perfect for some Abelian linear unital $\ell$-group $(H,u)$. We note that in \cite{DFL}, a lexicographic MV-algebra denotes an MV-algebra having a lexicographic ideal which will be defined below in Section \ref{section}. But by Theorem \ref{th:local1}, we will conclude that both notions are equivalent for symmetric pseudo MV-algebras from $\mathcal M$.

It is worthy to note that according to Example \ref{ex:3.3}, the pseudo MV-algebra $M$ has two isomorphic lexicographic representations $\Gamma(\mathbb Z \lex (\mathbb Z\lex \mathbb Z),(1,(0,0)))$ and $\Gamma((\mathbb Z \lex \mathbb Z)\lex \mathbb Z,((1,0),0))$, but $(H_1,u_1):=(\mathbb Z, 1)$ and $(H_2,u_2):=(\mathbb Z \lex \mathbb Z,(1,0))$ are not isomorphic, as well as $G_1:=\mathbb Z \lex \mathbb Z$ and $G_2:=\mathbb Z$ are not isomorphic $\ell$-groups.

\section{Local Pseudo MV-algebras with Retractive Radical}%5

In \cite[Cor 2.4]{DiLe2}, the authors characterized MV-algebras that can be expressed in the form $\Gamma(\mathbb H \lex G,(1,0))$, where $\mathbb H$ is a subgroup of $\mathbb R$ and $G$ is an Abelian $\ell$-group. In what follows, we extend this characterization for local symmetric pseudo MV-algebras. This result gives another characterization of strong $(\mathbb H,1)$-perfect pseudo MV-algebras via lexicographic product.

We denote by ${\mathcal  M}$   the set of pseudo MV-algebras $M$ such that
either every maximal ideal of $M$ is normal or $M$ is trivial. By
\cite[(6.1)]{DDT},  ${\mathcal  M}$ is a variety.

Let $M$ be a symmetric pseudo MV-algebra. For any $x\in M$, we define the {\it order}, in symbols $\ord(x)$, as the least integer $n$ such that $n. x =1$ if such $n$ exists, otherwise, $\ord(x)=\infty$. It is clear that the set of all elements with infinite order is an ideal. An element $x$ is {\it finite} if $\ord(x)<\infty$ and $\ord(x^-)<\infty$.

\begin{lemma}\label{le:loc1}
Let $M$ be a pseudo MV-algebra from $\mathcal M$ and $x \in M$. There exists a proper normal ideal of $M$ containing $x$ if and only if $\ord(x)=\infty$.
\end{lemma}

\begin{proof}
%Without loss of generality we can assume in view of Theorem \ref{th:2.1} that $M=\Gamma(G,u)$ for some unital $\ell$-group $(G,u)$ and symmetricity of $M$ entails that $u \in C(G)$.  If $x\in M$ and $a\in G$, then $a+x-a \in M$, where the addition and difference are taken as those in the group $G$.
%
%
%Let $x$ be any element of $M$ and let $N(x)$ be the normal ideal of $M$ generated by $x$. Then is easy seen that
%$$
%N(x)=\{y \in M: y \le a_1+x-a_1+\cdots+a_n+x-a_n,\ a_1,\ldots,a_n \in G,\ n\ge 1\}.$$

Let $x$ be any element of $M$ and let $I(x)$ be the normal ideal of $M$ generated by $x$. Then
$$
I(x)=\{y \in M: y\le m.x\ \mbox{ for some } m\in \mathbb N\}. \eqno(5.1)
$$

%Clearly, if $x$ belongs to a proper ideal, then $\ord(x)=\infty$. On the other hand, let $\ord(x)=\infty$. By (5.1), $1 \notin I(x)$. There exists a maximal ideal $I$ of $M$ containing $I(x)$ and since $M \in \mathcal M$, $I$ is normal and $x \in I$.
\end{proof}

\begin{lemma}\label{le:loc2}
Let $M$ be a symmetric pseudo MV-algebra. If $\ord(x\odot y)<\infty$, then $x\le y^-$.
\end{lemma}

\begin{proof}
By the hypothesis, $\ord(x\odot y)=n$ for some integer $n \ge 1.$ Hence $(y^-\oplus x^-)^n=0.$ By \cite[Prop 1.24(ii)]{GeIo}, $(y^-\oplus x^-)\vee (x\oplus y)=1$ which by \cite[Lem 1.32]{GeIo} yields $(y^-\oplus x^-)^n\vee (x\oplus y)^n=1$, so that $(x\oplus y)^n=1$ and $x\oplus y =1$, consequently, $x\le y^-$.
\end{proof}

\begin{lemma}\label{le:loc3}
Let $M\in \mathcal M$ be a symmetric pseudo MV-algebra. The following statements are equivalent:
\begin{enumerate}
\item[{\rm (i)}] $M$ is local.
\item[{\rm (ii)}] For every $x \in M$, $\ord(x)<\infty$ or $\ord(x^-)<\infty$.

\end{enumerate}
\end{lemma}

\begin{proof}
Let $M$ be local. There exists a unique maximal ideal $I$ that is normal. Assume that for some $x\in M$, we have $\ord(x)=\infty=\ord(x^-)$. By Lemma \ref{le:loc1}, $x,x^-\in I$ which is absurd.

Conversely, let for every $x \in M$, $\ord(x)<\infty$ or $\ord(x^-)<\infty.$ Let $I$ be a maximal ideal of $M$ and assume that $x \notin M$ for some $x\in M$ with $\ord(x)=\infty.$ Since $I$ is by the hypothesis normal, by a characterization of normal and maximal ideals, \cite[Prop 3.5]{GeIo}, there is an integer $n\ge 1$ such that $(x^-)^n\in I$. By Lemma \ref{le:loc1}, $\ord((x^-)^n)=\infty$ and $\ord(((x^-)^n)^-) <\infty$, i.e. $\ord(n.x)<\infty$, which implies $\ord(x)<\infty$ that is impossible. Hence, every element $x$ with infinite order belongs to $I$, and so $I$ is a unique maximal ideal of $M$, in addition $I$ is normal.
\end{proof}

\begin{lemma}\label{le:loc4}
Let $M\in \mathcal M$ be a local symmetric pseudo MV-algebra and let $I$ be a unique maximal ideal of $M$.  For all $x,y \in M$ such that $x/I \ne y/I$, we have $x<y$ or $y<x$.
\end{lemma}

\begin{proof}
By hypothesis, we have that $x\odot y^-\notin I$ or $y\odot x^-\notin I$. By Lemma \ref{le:loc3}, in the first case we have $\ord(x\odot y^-)<\infty$ which by Lemma \ref{le:loc2} implies $x\le y$ and consequently $x<y$. In the second case, we similarly conclude $y<x$.
\end{proof}

We introduce the following notion. A normal ideal $I$ of a pseudo MV-algebra $M$ is said to be {\it retractive} if the canonical projection $\pi_I: M  \to M/I$ is retractive, i.e. there is a homomorphism $\delta_I: M/I \to M$ such that $\pi_I\circ \delta_I=id_{M/I}$. If a normal ideal $I$ is retractive, then $\delta_I$ is injective and $M/I$ is isomorphic to a subalgebra of $M$.

For example, if $M=\Gamma(H\lex G,(u,0))$ and $I=\{(0,g): g \in G^+\}$, then $I$ is a normal ideal, see Theorem \ref{th:3.2}(vi), and due to  $M/I \cong \Gamma(H,u) \cong \Gamma(H\lex \{0\},(u,0)) \subseteq \Gamma(H\lex G,(u,0))$, $I$ is retractive.

\begin{lemma}\label{le:ret1}
Let $I$ be a normal ideal of a symmetric pseudo MV-algebra. Then the following are equivalent:
\begin{enumerate}
\item[{\rm (i)}] $x/I = y/I$.
\item[{\rm (ii)}] $x= (h\oplus y)\odot k^-$, where $h,k \in I$.
\end{enumerate}
\end{lemma}

\begin{proof}
(i) $\Rightarrow$ (ii) Assume $x/I=y/I$. Then the elements $k=x^-\odot y$ and $h=x\odot y^-$ belong to $I$. It is easy to see that $x\oplus k=x\vee y=h\oplus y$. Since $k^-=y^-\oplus x \ge x$, we have $x = x\wedge k^-= (x\oplus k)\odot k^-= (h\oplus y)\odot k^-$.

(ii) $\Rightarrow$ (i) Then we have $x/I=y/I$.
\end{proof}

Let $M$ be a pseudo MV-algebra, and let $\mbox{Sub}(M)$ be the set of all subalgebras of $M$. Then $\mbox{Sub}(M)$ is a lattice with respect to set theoretical inclusion with the smallest element  $\{0,1\}$ and greatest one $M$. It is easy to see that if $M$ is symmetric and $I$ is an ideal of $M$, then the subalgebra  $\langle I\rangle$ of $M$ generated by $I$ is the set $\langle I \rangle =I \cup I^-.$ We recall a subalgebra $S$ of $M$ is said to be a {\it complement} of a subalgebra $A$ of $M$ if $S\cap A=\{0,1\}$ and $S\vee A=M$.

In the following, we characterize retractive ideals of pseudo MV-algebras in an analogous way as it was done for MV-algebras in \cite[Thm 1.2]{CiTo}.

\begin{theorem}\label{th:retract}
Let $M$ be a symmetric pseudo MV-algebra and $I$ a normal ideal of $M$. The following statements are equivalent:
\begin{enumerate}
\item[{\rm (i)}] $I$ is a retractive ideal.
\item[{\rm (ii)}] $\langle I \rangle$ has a complement.
\end{enumerate}
\end{theorem}

\begin{proof}
Let $I$ be a retractive ideal of $M$ and let $\delta_I:M/I \to M$ be an injective homomorphism such that $\pi_I\circ \delta_I=id_{M/I}$. We claim that $\delta_I(M/I)$ is a complement of $\langle I \rangle$. Clearly $\delta_I(M/I) \cap \langle I \rangle =\{0,1\}$.  Let $x \in M$, then $x/I = \delta_I(M/I)(x/I)/I$ so that by Lemma \ref{le:ret1}, we have $x = (h\oplus \delta_I(x/I))\odot k^-$ for some $h,k \in I$ that implies $x \in \delta_I(M/I)\vee \langle I \rangle$.

Conversely, assume that $\langle I \rangle$ has a complement $S\in \mbox{Sub}(M)$. From $S \cap \langle I \rangle=\{0,1\}$ we conclude that the canonical projection $\pi_I$ is injective on $S$. Indeed, if for $x,y \in S$, we have $x/I=y/I$, then also $x/I=(x\vee y)/I=y/I$ which yields $(x\vee y)\odot x^-\in S\cap I=\{0\}$, $(x\vee y)\odot y^-\in S\cap I=\{0\}$. Therefore, $x=x\vee y=y$ and this implies that the restriction $\pi_{I\upharpoonleft S}$ is injective.

From $S \vee \langle I \rangle=M$, we have that for each $x \in M$, there is a term in the language of pseudo MV-algebras, says $p(a_1,\ldots,a_m,b_1,\ldots, b_n)$, such that
$$ x = p^M(x_1,\ldots,x_m,y_1,\ldots, y_n)
$$
for some $x_1,\ldots,x_m \in S$ and $y_1,\ldots,y_n \in \langle I \rangle$. Then
$$ x/I = p^{M/I}(x_1/I,\ldots,x_m/I,y_1/I,\ldots, y_n/I).
$$
Since $y_i/I \in \{0,1\}$ for each $i=1,\ldots,n$, there is an $n$-tuple $(t_1,\ldots,t_n)$ of elements from $\{0,1\}$ such that
$x/I = p^{M}(x_1,\ldots,x_m,t_1,\ldots,t_n)/I$.  Since
$$
(x_1,\ldots,x_m,t_1,\ldots,t_n) \in S^{m+n},
$$
we have that $p^{M}(x_1,\ldots,x_m,t_1,\ldots,t_n) \in S$. Therefore, the restriction $\pi_{I\upharpoonleft S}$ is an isomorphism from $S$ onto $M/I$, and setting $\delta_I = (\pi_{I\upharpoonleft S})^{-1}$, we see that $I$ is retractive.
\end{proof}

\begin{theorem}\label{th:local}
Let $M$ be a symmetric pseudo MV-algebra from $\mathcal M$. The following statements are equivalent:

\begin{enumerate}
\item[{\rm (i)}] $M$ is local and $\Rad_n(M)$ is retractive.
\item[{\rm (ii)}] $M$ is strong $(\mathbb H,1)$-perfect for some subgroup $\mathbb H$ of $\mathbb R$ with $1 \in \mathbb H$.
\item[{\rm (iii)}] There exists a subgroup $\mathbb H$ of $\mathbb R$ with $1 \in \mathbb H$ and an $\ell$-group $G$ such that $M \cong \Gamma(\mathbb H\lex G, (u,0))$.
%\item[{\rm (iv)}] $M$ is $(\mathbb H,1)$-perfect pseudo MV-algebra with an $(\mathbb H,1)$-strong cyclic property.
\end{enumerate}
\end{theorem}

\begin{proof}
(i) $\Rightarrow$ (ii)  Let $I$ be a unique maximal and normal ideal of $M$ and let $(K,v)$ be a (unique up to isomorphism) unital $\ell$-group given by Theorem \ref{th:2.1}, such that $M \cong \Gamma(K,v)$; without loss of generality we can assume that $M=\Gamma(K,v)$. By \cite{156}, there is an extremal state (= state morphism) $s_0: M\to [0,1]$ such that $I=\Ker(s_0)$. The range of $s_0$, $s_0(M)$, is an MV-algebra which corresponds to a unique subgroup $\mathbb H$ of $\mathbb R$ such that $s_0(M) = \Gamma(\mathbb H,1)$ is a subalgebra of $\Gamma(\mathbb R,1)$.

Since  $I=\Rad_n(M)$, $I$ is a retractive ideal, and $M/I$ is isomorphic to $\Gamma(\mathbb H,1)$, we have $\Gamma(\mathbb H,1)$ can be injectively embedded into $K$ and $\mathbb H$ is isomorphic to a subgroup of $K$.

In addition, let $\langle I\rangle$ be a subalgebra of $M$ generated by $I$. Then $\langle I\rangle = I \cup I^-=I \cup I^-$, $I^-=I^\sim$, and $\langle I\rangle$ is a perfect pseudo MV-algebra. By \cite[Prop 5.2]{DDT}, there is a unique (up to isomorphism) $\ell$-group $G$ such that $\langle I\rangle \cong \Gamma(\mathbb Z\lex G,(1,0))$.

In what follows, we prove that $M \cong \Gamma(\mathbb H\lex G,(1,0))$.

Define $M_t=s^{-1}(\{t\})$, $t \in [0,1]_\mathbb H$. We assert that $(M_t: t \in [0,1]_\mathbb H)$ is an $(\mathbb H,1)$-decomposition of $M$. It is clear that it is a decomposition: Every $M_t$ is non-empty, and $M_t^- =M_{1-t}=M_t^\sim$ for each $t\in [0,1]_\mathbb H$. In addition, if $x\in M_v$ and $y \in M_t$, then $x\oplus y \in M_{v\oplus t}$, $x\wedge y \in M_{v\wedge t}$ and $x\vee y  \in M_{v\vee t}$. By Lemma \ref{le:loc4}, we have $M_s \leqslant M_t$ for all $ s<t$, $s,t \in [0,1]_ \mathbb H$.

Since $I=\Rad_n(M)$ is retractive, there is a unique subalgebra $M'$ of $M$ such that $s_0(M')=s_0(M)$. For any $t\in [0,1]_\mathbb H$, there is a unique element $x_t \in M'$ such that $s_0(x_t)=t$. We assert that the system $(x_t \colon t \in [0,1]_\mathbb H)$ satisfies the following properties (i) $c_t \in M_t$ for each $t \in [0,1]_\mathbb H$, (ii) $c_{v+t}=c_v+c_t$ whenever $v+t\le 1$, (iii) $c_1=1$. (iv) $c_t\in C(K)$. Indeed, since $s_0$ is a homomorphism of pseudo MV-algebras, by the categorical equivalence Theorem \ref{th:2.1}, $s_0$  can be uniquely extended to a unital $\ell$-group homomorphism $\hat s_0: (K,v)\to (\mathbb H,1)$. Now
if $x$ is any element of $K$, then $x+c_t-x\in M$ because $M$ is symmetric, and hence $\hat s_0(x+c_t-x)=\hat s_0(x) + \hat s_0(c_t)- \hat s_0(x)=s_0(c_t)=t$ which implies $x+ c_t-x= c_t$ so that $x+c_t=c_t+x$.

In other words, we have proved that $(M_t: t \in [0,1]_\mathbb H)$ has the strong cyclic property,  and consequently, $M$ is strong $(\mathbb H,1)$-perfect. By Theorem \ref{th:3.5}, $M \cong \Gamma(\mathbb H\lex G',(1,0))$ for some unique (up to isomorphism) $\ell$-group $G'$. Hence $G' \cong G$, where $G$ was defined above,  which proves (ii) $\Rightarrow$ (iii).

The implication (iii) $\Rightarrow$ (i) is evident by the note that is just before Theorem \ref{th:local}.
%Finally, it is evident that (ii) and (iv) are equivalent.
\end{proof}

We note that if $M$ is a local symmetric pseudo MV-algebra with a retractive ideal $\Rad_n(M)$, then $M$ is a lexicographic extension of $\Ker_n(M)$ in the sense described in \cite{HoRa}.

\begin{proposition}\label{pr:product}
Let $(M_\alpha: \alpha \in A)$ be a system of pseudo MV-algebras and let $I_\alpha$ be a non-trivial normal ideal of $M_\alpha$, $\alpha \in A$. Set $M=\prod_\alpha M_\alpha$ and $I=\prod_\alpha I_\alpha$. Then $I$ is a retractive ideal of $M$ if and only if every $I_\alpha$ is a retractive ideal of $M_\alpha$.
\end{proposition}

\begin{proof} The set $I=\prod_\alpha I_\alpha$ is a non-trivial normal ideal of $M$. Then $M/I \cong \prod_\alpha M_\alpha/I_\alpha$ and without loss of generality, we can assume that $M/I =\prod_\alpha M_\alpha/I_\alpha$.

Assume that every $I_\alpha$ is retractive. We denote by $\pi_\alpha$ the canonical projection of $M_\alpha$ onto $M_\alpha/I_\alpha$ and by $\delta_\alpha:M_\alpha/I_\alpha \to M_\alpha$ its right inversion i.e. $\pi_\alpha \circ \delta_\alpha = id_{M_\alpha/I_\alpha}$. Let $\pi:M\to M/I$ be the canonical projection. If we set $\delta:M/I\to M$ by $\delta((x_\alpha/I_\alpha)_\alpha):= ((\delta_\alpha(x_\alpha/I_\alpha))_\alpha)$, then we have $\pi\circ \delta = id_{M/I}$, so that $I$ is retractive.

Conversely, let $I$ be a retractive ideal of $M$. Let $\pi^\alpha:\prod_\alpha M_\alpha$ be the $\alpha$-th projection of $M$ onto $M_\alpha$. We define a mapping $\delta_\alpha: M_\alpha/I_\alpha \to M_\alpha$ by $\delta_\alpha = \pi^\alpha \circ \delta$ $(\alpha \in A$). Then $\prod_\alpha \pi_\alpha \circ \delta_\alpha(x_\alpha/I_\alpha)= \prod \pi_\alpha \circ \pi^\alpha \circ \delta (x_\alpha/I_\alpha)$ which yields $\pi_\alpha \circ \delta_\alpha = id_{M_\alpha/I_\alpha}$.
\end{proof}

\begin{corollary}\label{co:product}
Let $I$ be a non-trivial normal ideal of a pseudo MV-algebra $M$ and let $\alpha$ be a cardinal. Then the power $I^\alpha$ is a retractive ideal of the power pseudo MV-algebra $M^\alpha$ if and only if $I$ is a retractive ideal of $M$.
\end{corollary}

\section{Free Product and Local Pseudo MV-algebras}%6

In the present section we show that every local pseudo MV-algebra that is a strong $(\mathbb H,1)$-perfect pseudo MV-algebra has also a representation via a free product. It will generalize results from \cite{DiLe2} known only for local MV-algebras.

Let $\mathcal V$ be a class of pseudo MV-algebras and let $\{A_t \}_{ t \in
T}\subseteq \mathcal V$. According to \cite{DvHo1}, we say that a $\mathcal V$-{\it coproduct} (or simply
a {\it coproduct} if $\mathcal V$ is known from the context) of this
family is a pseudo MV-algebra $A\in \mathcal V$,
together with a family
of homomorphisms

 $$
 \{f_t \colon  A_t\to A\}_{t \in T}
 $$
such that
\begin{enumerate}
\item[{\rm (i)}] $\bigcup_{t\in T} f_t(A_t)$ generates $A$;
\item[{\rm (ii)}] if $B\in \mathcal V$ and $\{g_t \colon \ A_t\to
B\}_{ t \in T}$ is a family of homomorphisms, then
there exists a (necessarily) unique homomorphism $h \colon
A\to B$ such that $g_t=f_th$ for all $t\in T$.
\end{enumerate}

Coproducts exist for every variety $\mathcal V$ of algebra, and are
unique. They are designated by $\bigsqcup_{t\in T}^{\mathcal V} A_t$
(or $A_1\sqcup^{\mathcal V}A_2$ if $T=\{1,2\}$). If each of the
homomorphisms $f_t$ is an embedding, then the coproduct  is called
the {\it free product}.

By \cite[Thm 2.3]{DvHo1}, the free product of any set of non-trivial pseudo MV-algebras exists in the variety of pseudo MV-algebras.

Now let $M$ be a symmetric local pseudo MV-algebra from
$\mathcal M$ with a unique maximal and normal ideal $I=\Ker_n(M)=\Ker(M)$. Let $\mathbb H$ be a subgroup of $\mathbb R$ such that $M/I\cong \Gamma(
\mathbb H,1)$. Take an $\ell$-group $G$ such that $\Gamma(\mathbb Z\lex G,(1,0))\cong \langle I \rangle$.  Let $N =\Gamma(\mathbb H \lex G,(1,0))$. If $I$ is retractive, then by Theorem \ref{th:local}, $M \cong N$, and in this section, we describe this situation using free product of $M/I$ and $\langle I\rangle$. We note that this was already established in \cite[Thm 3.1]{DiLe2} but only
for MV-algebras.  For our generalization, we introduce a weaker form of our free product of $M/I$ and $\langle I \rangle$ which we will denote $M/I\sqcup_w \langle I\rangle$ in the variety of symmetric pseudo MV-algebras from $\mathcal M$ and which means that (i) remains and (ii) are changed as follows
\begin{enumerate}
\item[(i*)] if $\phi_1: M/I \to M/I\sqcup_w \langle I\rangle$ and $\phi_2:\langle I \rangle \to M/I\sqcup_w \langle I\rangle$ are injecive homomorphisms, then $\phi_1(M/I)\cup \phi_2(\langle I\rangle)$ generates $M/I\sqcup_w \langle I\rangle$,
\item[(ii*)] if $\kappa_1:M/I\to A$ and $\kappa_2: \langle I \rangle \to A$, where $A$ is a symmetric pseudo MV-algebra from $\mathcal M$, are such homomorphisms that $\kappa_1(a)+ \kappa_2(b)=\kappa_2(b)+\kappa_1(a)$, then there is a unique homomorphism $\psi: M/I\sqcup_w \langle I\rangle\to A$ such that $\psi \circ \phi_1 = \kappa_1$ and $\psi\circ \phi_2 = \kappa_2$.
\end{enumerate}
We note that if $M$ is an MV-algebra, then our notion coincides with the original form of the free product of MV-algebras in the class of MV-algebras.

\begin{theorem}\label{th:coprod}
Let $M$ be a symmetric local pseudo MV-algebra from $\mathcal M$, $I =\Rad_n(I)$ and $N =\Gamma(\mathbb H\lex G, (1,0))$ for some unital $\ell$-subgroup $(\mathbb H,1)$ of $(\mathbb R,1)$ and some $\ell$-group $G$.  The following statements are equivalent:

\begin{enumerate}
\item[{\rm (i)}] $M \cong N$.
\item[{\rm (ii)}] The free product $M/I \sqcup_w \langle I\rangle$ in the variety of symmetric pseudo MV-algebras from $\mathcal M$ is isomorphic to  $M$.
\end{enumerate}
\end{theorem}

\begin{proof}
(i) $\Rightarrow$ (ii) Let $M=\Gamma(K,v)$. By Theorem \ref{th:local}, $I=\Ker_n(M)$ is a retractive ideal. Define $\phi_1: M/I \to \Gamma(\mathbb H \lex \{0\},(1,0))\subset \Gamma(\mathbb H \lex G,(1,0))=N$ and $\phi_2: \langle I \rangle \to \Gamma(\mathbb Z \lex G,(1,0))\subset \Gamma(\mathbb H \lex G,(1,0))=N$ as follows: Let $s_0$ be a unique state on $M$ which is guaranteed by Theorem \ref{th:3.4}. We set $M_t=s_0^{-1}(\{t\})$ for any $t \in [0,1]_\mathbb H$.
Then $\phi_1(x/I):=(t,0)$ whenever $x \in M_t$. Since $\langle I\rangle =I \cup I^-$, we set $\phi_2(x)=(0,x)$ if $x \in I$ and $\phi_2(x)=(1,x-1)$  if $x \in I^-$. From (4.2) of the proof of Theorem \ref{th:3.5} we see that $\phi_1$ and $\phi_2$ are injective homomorphisms of pseudo MV-algebras into $N$. Using again (4.2), we see that $\phi_1(M/I)\cup \phi_2(\langle I\rangle)$ generates $N$.

Now suppose that there is a symmetric pseudo MV-algebra $A$ from $\mathcal M$   and two mutually commuting homomorphisms $\kappa_1: M/I\to A$ and $\kappa_2: \langle I \rangle\to A=\Gamma(W,w)$, i.e. $\kappa_1(a)+\kappa_2(b)=\kappa_2(b)+\kappa_1(a)$ for all $a \in M/I$ and $b\in \langle I\rangle$.  Then $\kappa_1(1/I)=w=\kappa_2(1)$ and $w$ commutes with every $\kappa_1(a)$ and $\kappa_2(b)$.

\vspace{2mm}
\noindent{\bf Claim 1.} {\it Let $a=\kappa_1\phi_1^{-1}(h,0)$ with $0<h<1$, $h \in H$, and $\epsilon =\kappa_2\phi_2^{-1}(0,g)$ with $g \in G^+$.  Then $\epsilon < a< \epsilon^-$.}

\vspace{2mm}
Indeed, by the assumption, from the form of the element $a$ we conclude that it is finite and $\epsilon$ and $a$ commute. Then there is an integer $n \ge 1$ such that $n.a=1$. Since $\epsilon \in \Rad(A)$, we have $n.\epsilon = n\epsilon <1=n.a \le na$ which yields $0\le n(a-\epsilon)$, so that $\epsilon < a$. In a similar way we show $\epsilon <a^-$, i.e. $\epsilon < a <\epsilon^-$.

\vspace{2mm}
\noindent{\bf Claim 2.} {\it Let $\alpha =\kappa_1\circ \phi^{-1}:\Gamma(\mathbb H\lex \{0\},(1,0))\to A$ and $\beta=\kappa_2\circ \phi_2^{-1}: \phi_2^{-1}(\langle I \rangle)\to A$. Passing to the corresponding representing unital $\ell$-groups, we will denote by $\hat\alpha$ and $\hat\beta$ the corresponding extensions of $\alpha$ and $\beta$ to $\ell$-homomorphisms of unital $\ell$-groups into the unital $\ell$-group $(G_A,w)$ such that $\Gamma(G_A,w)=A$.  Then $\hat\alpha(0,h)+\hat\beta(0,g)\ge 0$ for each $h \in \mathbb H^+$ and each $g\in G$.  }

\vspace{2mm}
If $h=0$, the statement is evident.  Let $h>0$.  Then $a:=\hat\alpha(0,h)+\hat\beta(0,g)= \hat\alpha(h,0)+\hat\beta(0,g^+) + \hat\beta(0,g^-)$, where $g^+=g\vee 0$ and $g^-=g\wedge 0$. Then $a =\hat\alpha(h,0)+\beta(0,g^+) +\beta(1,g^-)-\beta(1,0)$. From Claim 1, we get $\hat\alpha(h,0)+\beta(0,g^+) \ge \beta(0,-g^-)=\beta(1,0)-\beta(1,g^-)$ and the claim is proved.

\vspace{2mm}
Now we define a mapping $\psi: \mathbb H\lex G \to G_A$ by
$$\psi(h,g)=\hat \alpha(h,0)+\hat\beta(0,g),\quad (h,g)\in \mathbb H\lex G.
$$

\noindent{\bf Claim 3.} {\it $\psi$ is an $\ell$-group homomorphism of unital $\ell$-groups.}
\vspace{2mm}

(a) We have $\psi(0,0)=0$ and $\psi(1,0)=w$. Moreover,
\begin{eqnarray*}
\psi(h_1,g_1)+\psi(h_2,g_2)&=& \hat \alpha(h_1,0)+\hat\beta(0,g_1)+ \hat \alpha(h_2,0)+\hat\beta(0,g_2)\\
&=& \hat \alpha(h_1,0)+ \hat \alpha(h_2,0)+ \hat\beta(0,g_1) +\hat\beta(0,g_2)\\
&=& \hat \alpha(h_1+h_2,0)+\hat\beta(0,g_1+g_2)\\
&=& \psi(h_1+h_2,g_1+g_2).
\end{eqnarray*}

(b) According to Claim 2, we see that $\psi(h,g)\ge 0$ whenever $(h,g)\ge (0,0)$.

(c) $\psi$ preserves $\wedge$. For $x:=(h_1,g_1)\wedge (h_2,g_2)$, we have three cases (i) $x= (h_1,g_1)$ if $h_1 < h_2$, (ii) $x= (h_1,g_1\wedge g_2)$ if $h_1=h_2$, and (iii) $x=(h_2,g_2)$ if $h_2 < h_1.$

In  case (i), we have $\psi(h_2,g_2)-\psi(h_1,g_1)=\psi(h_2-h_1,g_2-g_1)\ge 0$ by Claim 2. Thus $\psi$ preserves $\wedge$.  In case (ii), we have \begin{eqnarray*}
\psi((h_1,g_1)\wedge (h_2,g_2))&=&\psi(h_1, g_1\wedge g_2)= \hat\alpha(h_1,0)+ \hat\beta(0,g_1\wedge g_2)\\
&=& \hat\alpha(h_1,0)+\hat\beta(0,g_1)\wedge \hat\beta(0,g_2)\\
&=& (\hat\alpha(h_1,0)+\hat\beta(0,g_1))\wedge (\hat\alpha(h_1,0)+\hat\beta(0,g_2))\\
&=& \psi(h_1,g_1)\wedge \psi(h_2,g_2).
\end{eqnarray*}
Case (iii) follows from (i).

If we restrict $\psi$ to $N$, then we have
$$\psi(h,g)= (\alpha(h,0)\oplus \beta(1,g^+))\odot \beta(1,g^-), \quad (h,g)\in N.
$$
Using that $\psi$ is an $\ell$-group homomorphism, we have that if $g=g_1+g_2$, where $g_1\ge0$ and $g_2\le 0$, then

$$\psi(h,g)= (\alpha(h,0)\oplus \beta(1,g_1))\odot \beta(1,g_2).
$$

\vspace{2mm}
{\it Uniqueness of $\psi$.} If $\psi'$ is another homomorphism from $N$ into $A$ such that $\phi_i\circ\psi = \kappa_i$ for $i=1,2$,
then $\psi'(0,0)=\psi(0,0)$, $\psi'(0,g)=\psi'(\phi_2\phi_2^{-1}(0,g)) =\phi_2 \kappa_2(0,g)=\psi(0,g)$, $g \in G^+$. $\psi'(h,0)=\psi'(\phi_1\phi_1^{-1}(h,0))= \phi_1\kappa_1(h,0)= \psi(h,0)$, $h\in [0,1]_\mathbb H$.

Using all above steps, we have that the free product $M/I \sqcup_w \langle I\rangle \cong N$.  Since $N\cong M$, we have established (ii).

(ii) $\Rightarrow$ (i)
From the proof of the previous implication we have that the free product of $\Gamma(\mathbb H,1)$ and $\Gamma(\mathbb Z \lex G,(1,0))$ is isomorphic to $N=\Gamma(\mathbb H \lex G,(1,0))$. Since $M/I \cong \Gamma(\mathbb H,1)$ and $\langle I \rangle \cong \Gamma(\mathbb Z \lex G,(1,0))$, we have from (ii) $M\cong N$.
\end{proof}

\section{Pseudo MV-algebras with Lexicographic Ideals}\label{section}%7

The following notions were introduced in \cite{DFL} only for MV-algebras, and in this section,  we extend them for symmetric pseudo MV-algebras and generalize some results from \cite{DFL}.

We say that a normal ideal $I$ is (i) {\it commutative} if $x/I\oplus y/I=y/I\oplus x/I$ for all $x,y \in M$, (ii) {\it strict} if $x/I < y/I$ implies $x<y$.

For example, (i) if $s$ is a state, then $\Ker(s)$ is a commutative ideal, \cite[Prop 4.1(ix)]{156}, (ii) every maximal ideal that is normal is commutative, \cite{156}. If $M$ is a local symmetric pseudo MV-algebra, $\Rad_n$ is a strict ideal.

Now we extend for pseudo MV-algebras the notion of a lexicographic ideal introduced in \cite{DFL} only for MV-algebras. We say that a commutative ideal $I$ of a pseudo MV-algebra $M$, $\{0\}\ne I \ne M$, is {\it lexicographic} if
\begin{enumerate}
\item[{\rm (i)}] $I$ is strict,
\item[{\rm (ii)}] $I$ is retractive,
\item[{\rm (iii)}] $I$ is prime.
\end{enumerate}

We note that a lexicographic ideal for MV-algebras was defined in \cite{DFL} by (i)--(iii) and

\begin{enumerate}
\item[{\rm (iv)}] $y\le x \le y^-$ for all $y \in I$ and all $x \in M \setminus \langle I\rangle$, where $\langle I \rangle $ is the subalgebra of $M$ generated by $I$.
\end{enumerate}

But since $I$ is strict, we have $y \in I^-$ implies $z<y$ for any $z \in I$. Hence, if $z \notin I$, we have $z/I > x/I=0/I$ for all $x \in I$ which yields $z>x$.  Therefore, $\langle I \rangle =I \cup I^-$ and (iv) holds, and consequently, (iv) from \cite{DFL} is superfluous, and for the definition of a lexicographic ideal of an MV-algebra we need only (i)--(iii).

Let $\mbox{LexId}(M)$ be the set of lexicographic ideals of $M$. If we take the MV-algebra $M$ from Example \ref{ex:3.3}, we see that $I_1=\{(0,m,n): m > 0, n \in \mathbb Z \mbox{ or } m=0, n\ge 0\}$ and $I_2=\{(0,0,n): n \ge 0\}$ are two unique lexicographic ideals of $M$ and $I_1 \subset I_2$.

\begin{proposition}\label{pr:loc5}
If $I,J\in \mbox{\rm LexId}(M)$, then $I\subseteq J$ or $J\subseteq I$. In addition, every lexicographic ideal is contained in the radical $\Rad(M)$ of $M$. If one of the lexicographic ideals is a maximal ideal, then $M$ has a unique maximal ideal of $M$.

%In addition, if every normal ideal of $M$ is retractive, then $\mbox{\rm LexId}(M)$ has the greatest ideal $I$ which is the unique maximal ideal of $M$, $M$ admits a unique state $s$ and $\Ker(s)=I.$
\end{proposition}

\begin{proof}
Suppose the converse, that is, there are $x \in I\setminus J$ and $y \in J\setminus I$. Then $x/I<y/I$ and $y/J<x/J$ which yields $x<y$ and $y<x$ which is absurd.

Assume that $I$ is a lexicographic ideal of $M$. If $I=\Rad(M)$, the statement is evident. If there is an element $y \in \Rad(M)$ such that $y\notin I$, then by (ii) $x<y$ for any element $x \in I$, so that $I \subseteq \Rad(M)$.

Let $I$ be any lexicographic ideal of $M$. We have two cases. (a) $I$ is a maximal ideal of $M$. We claim $M$ has a unique maximal ideal. Indeed, for any maximal ideal $J$ of $M$, $J \ne I$, there are $x \in I\setminus J $ and $y \in J\setminus I$ which implies $x <y$ so that $x \in J$ which is a contradiction. Hence, $I$ is a unique maximal ideal of $M$, then $\Rad(M)=I$ and every lexicographic ideal of $M$ is in $\Rad(M)$. (b)  $I$ is not a maximal ideal of $M$. Let $J$ be an arbitrary maximal ideal of $M$. There exists $y \in J\setminus I$ which yields $y>x$ for any $x \in I$, so that $x \in J$ and $I\subseteq J$. Hence, again $I \subseteq \Rad(M)$.
\end{proof}

\begin{remark}\label{re:1}
{\rm It is clear that if $\mbox{LexId}(M)\ne \emptyset$ is finite, then $\mbox{LexId}(M)$ has the greatest element. If $\mbox{LexId}(M)$ is infinite, we do not know whether $\mbox{LexId}(M)$ has the greatest element. And if this element exists, is it a maximal ideal of $M$?}
\end{remark}

We note that in Theorem \ref{th:variety}(1), we show that if $M$ is symmetric from $\mathcal M$ and $\mbox{LexId}(M)\ne \emptyset$, then $M$ is local.

As an interesting corollary we have the following statement.

\begin{corollary}\label{co:lex}
If $\mbox{\rm LexId}(M)$ is non-empty and $s$ is a state on $M$, then $s$ vanishes on each lexicographic ideal of $M$.
\end{corollary}

\begin{proof}
Let $I$ be a lexicographic ideal of $M$. First let $s$ be an extremal state. Then $\Ker(s)$ is by \cite[Prop 4.3]{156} a maximal ideal. Hence, by Proposition \ref{pr:loc5}, we have $I \subseteq \Ker(M) \subseteq \Ker(s)$, so that each extremal state vanishes on $I$. Therefore,  each convex combination of extremal states, and by Krein--Mil'man Theorem,  each state on $M$ vanishes on $I$.
\end{proof}

A strengthening of the latter corollary for lexicographic pseudo MV-algebras $M$ from $\mathcal M$ will be done in Corollary \ref{co:state2} showing that then $M$ has a unique state.

Now we present a prototypical examples of a pseudo MV-algebra with lexicographic ideal.

\begin{proposition}\label{pr:loc6}
Let $(H,u)$ be an Abelian linear unital $\ell$-group and let $G$ be an $\ell$-group. If we set $I=\{(0,g): g \in G^+\}$, then $I$ is a lexicographic ideal of $M=\Gamma(H\lex G,(u,0))$.

In addition, $M$ is subdirectly irreducible if and only if $G$ is a subdirectly irreducible $\ell$-group.
\end{proposition}

\begin{proof}
It is clear that $I$ is a normal ideal of $M$ as well as it is prime.

We have $x/I =0/I$ iff $x \in I$.
Assume $(0,g)/I < (h,g')/I$. Then $(h,g) \notin I$ that yields $h>0$ and $(0,g)<(h,g')$. Hence, if $x/I <y/I$, then $(y-x)/I > 0/I$ and $y-x>0$ and $x<y$.

Since $M/I\cong\Gamma(H\lex \{0\},(u,0)) \subseteq \Gamma(H\lex G,(u,0))$, we see that $I$ is retractive. Finally, let $y \in I$ and  $x \in M \setminus \langle I\rangle$. Then $\langle I \rangle = I \cup I^-$ and $x = (h,g')$ for some $h$ with $0<h<u$ and $g'\in G$. Then $y=(0,g)$ and hence, $y<x<y^-$.

The statement on subdirect irreducibility follows from the categorical representation of pseudo MV-algebras, Theorem \ref{th:2.1}.
\end{proof}

\begin{theorem}\label{th:local1}
Let $M$ be a symmetric pseudo MV-algebra from $\mathcal M$ and let $I$ be a lexicographic ideal of $M$. Then there is an Abelian linear unital $\ell$-group $(H,u)$ and an $\ell$-group $G$ such that $M \cong \Gamma(H\lex G,(u,0))$.
\end{theorem}

\begin{proof}
Similarly as in the proof of Theorem \ref{th:local}, we can assume that $M=\Gamma(K,v)$ for some unital $\ell$-group $(K,v)$. Since $I$ is lexicographic, then $I$ is normal and prime, so that $M/I$ is a linear, and since $I$ is also commutative, $M/I$ is an MV-algebra.  There is an Abelian linear unital $\ell$-group $(H,u)$ such that $M/I\cong \Gamma(H,u)$.

Let $\pi_I:M\to M/I$ be the canonical projection. For any $t \in [0,u]_H$, we set $M_t:=\pi_I^{-1}(\{t)\}$. We assert that $(M_t: t \in [0,u]_H)$ is an $(H,u)$-decomposition of $M$. Indeed, (a) let $x \in M_s$ and $y \in M_t$ for $s<t,$ $s,t \in [0,u]_H$. Then $\pi_I(x)=s <t<\pi(y)$ and $x<y$ because $I$ is strict. (b) Since $\pi_I$ is a homomorphism, $M_t^-=M_{u-t}=M_t^\sim$
for each $t \in [0,u]_H$. (c) Let $x\in M_s$ and $y \in M_t$, then $\pi_I(x\oplus y)=\pi_I(x)\oplus \pi_I(y)=s\oplus t.$

In addition, $\langle I\rangle = I \cup I^-=I \cup I^-$, $I^-=I^\sim$, and $\langle I\rangle$ is a perfect pseudo MV-algebra. By \cite[Prop 5.2]{DDT}, there is a unique (up to isomorphism) $\ell$-group $G$ such that $\langle I\rangle \cong \Gamma(\mathbb Z\lex G,(1,0))$.

Now we show that $(M_t: t \in [0,u]_H)$ has the strong cyclic property. Being $I$ also retractive, there is a subalgebra $M'$ of $M$ such that $M'\cong M/I$ and $\pi_I(M')=\pi_I(M)$. Then $M'$ is in fact an MV-algebra. For any $t \in [0,u]_H$, there is a unique $c_t \in M_t$ such that $\pi_I(c_t)=t.$ We assert that the system of elements $(c_t: t \in [0,u]_H)$ has the following properties: (i) $c_t \in M_t$, (ii) if $s+t \le u$, then $c_s+c_t \in M$ and $c_s+c_t=c_{s+t}$, (iii) $c_1=1$, and (iv) $c_t \in C(K)$ for each $t\in [0,u]_H$; indeed let $x \in K$. Being $M$ symmetric, the element $x +c_t - x\in H$ belongs also to $M$.  Due to the categorical equivalence, Theorem \ref{th:2.1}, the homomorphism $\pi_I$ can be uniquely extended to a homomorphism $\hat \pi_I: (K,v)\to (H,u)$ of unital $\ell$-groups. Hence, $\pi_I(x+c_t-x)= \hat \pi_I(x +c_t-x)= \hat\pi_I(x) +\hat\pi_I(c_t)-\hat\pi_I(x)= \pi_I(c_t)=t$ which implies $c_t = x +c_t-x$ and $x+c_t = c_t+x$.

Consequently, $M$ is a strong $(H,u)$-perfect pseudo MV-algebra. By Theorem \ref{th:3.5}, there is an $\ell$-group $G'$ such that $M\cong \Gamma(H \lex G',(u,0))$. By uniqueness (up to isomorphism of $\ell$-groups) of $G'$ in Theorem \ref{th:3.5}, we have $G' \cong G$ and consequently $M \cong \Gamma(H\lex G,(u,0))$.
\end{proof}

According to the latter theorem and Proposition \ref{pr:loc6}, we see that our notion of a lexicographic pseudo MV-algebra for symmetric pseudo MV-algebras from $\mathcal M$ coincides with the notion of one defined for MV-algebras in \cite{DFL} as those having at least one lexicographic ideal.

In the following result we compare the class of local pseudo MV-algebras with the class of lexicographic pseudo MV-algebras.

\begin{theorem}\label{compar}
{\rm (1)} The class of lexicographic pseudo MV-algebras from $\mathcal M$ is strictly included in the class of symmetric local pseudo MV-algebras.

{\rm (2)} The class of symmetric local pseudo MV-algebras with retractive radical is strictly included in the class of lexicographic pseudo MV-algebras from $\mathcal M$.
\end{theorem}

\begin{proof}
(1)  Let $M$ be a lexicographic pseudo MV-algebra from $\mathcal M$. By Theorem \ref{th:local1}, $M$ is symmetric and it is isomorphic to some $M':=\Gamma(H\lex G,(u,0))$, where $(H,u)$ is an Abelian unital $\ell$-group and $G$ is an $\ell$-group. Then the ideal $I=\{(0,g): g \in G^+\}$ is by Proposition \ref{pr:loc6} a retractive ideal of $M'$. By Proposition \ref{pr:loc5}, we have $I \subseteq \Rad(M')=\Rad_n(M')$. Since $I$ is prime, so is $\Rad_n(M')$ which yields $M'/\Rad_n(M')$ is linearly ordered and semisimple. Hence, $M'/\Rad_n(M')$ is a simple MV-algebra. Therefore, by \cite[Prop 3.3-3.5]{156}, $\Rad_n(M)$ is a maximal ideal which yields that $M'$ is local and, consequently $M$ is local.

To show that  the class of lexicographic pseudo MV-algebras from $\mathcal M$ is strictly included in the class of symmetric local pseudo MV-algebras, we can use an example from the proof of \cite[Thm 4.7]{DFL} or the pseudo MV-algebra $\Gamma(\mathbb Z \lex \mathbb Z,(2,1))$ that has no lexicographic ideal.

(2) By Theorem \ref{th:local}, we get that the class of symmetric local pseudo MV-algebras with retractive radical is strictly included in the class of lexicographic pseudo MV-algebras from $\mathcal M$. Using an example from \cite[Thm 4.7]{DFL}, we conclude that this inclusion is proper.
\end{proof}

The latter result entails the following corollary.

\begin{corollary}\label{co:state2}
Every lexicographic pseudo MV-algebra from $\mathcal M$ admits a unique state.
\end{corollary}

\begin{proof}
If $M$ is a lexicographic pseudo MV-algebra from $\mathcal M$, by (i) of Theorem \ref{compar}, we see that $M$ is local, that is, it has a unique maximal ideal and this ideal is normal. Due to a one-to-one relation between extremal states and maximal and normal ideals of $M$, \cite{156}, we conclude $M$ admits a unique state.
\end{proof}

The following result gives a new look to Theorem \ref{th:local}.

\begin{theorem}\label{compar1}
Let $M$ be a lexicographic symmetric pseudo MV-algebra from $\mathcal M$. The following statements are equivalent:
\begin{enumerate}
\item[{\rm (i)}] $\Rad_n(M)$ is a lexicographic ideal.
\item[{\rm (ii)}] $M$ is strongly $(\mathbb H,1)$-perfect for some unital $\ell$-subgroup $(\mathbb H,1)$ of $(\mathbb R,1)$.
\end{enumerate}
\end{theorem}

\begin{proof}
Let $M\cong \Gamma(H\lex G,(u,0))$ for some Abelian unital $\ell$-group $(H,u)$ and an $\ell$-group $G$ and let $I$ be a retractive ideal of $M$ such that $M/I\cong \Gamma(H,u)$. By Proposition \ref{pr:loc6}, $I \subseteq \Rad_n(M)$.

(i) $\Rightarrow $ (ii)
If $\Rad_n(M)$ is a retractive ideal, then $M/\Rad_n(M)$ is a semisimple MV-algebra that is linearly ordered because $\Rad_n(M)$ is a prime normal ideal. Again applying  by \cite[Prop 3.4-3.5]{156}, $M/\Rad_n(M) \cong \Gamma(H,u)$ and $\Gamma(H,u)$ is isomorphic to some $(\mathbb H,1)$.

(ii) $\Rightarrow $ (i) Since $M/I \cong \Gamma(\mathbb H,1)$,
as a consequence of \cite[Prop 3.4-3.5]{156}, we get $I$ is a maximal ideal of $M$. Hence, $I =\Rad_n(M)$ and $I$ is a lexicographic ideal of $M$ and $M\cong \Gamma(\mathbb H\lex G,(1,0))$.
\end{proof}

We say that a pseudo MV-algebra $M$ from $\mathcal M$ is $I$-{\it representable} if $I$ is a lexicographic ideal of $M$ and $M\cong \Gamma(H\lex G, (u,0))$, where $(H,u)$ is an Abelian unital $\ell$-group such that $M/I \cong \Gamma(H,u)$ and $G$ is an $\ell$-group such that $\langle I\rangle \cong \Gamma(\mathbb Z \lex G,(1,0))$;  the existences of $(H,u)$ and $G$ are guaranteed by Theorem \ref{th:local1}.

\begin{theorem}\label{th:variety}
The class of lexicographic pseudo MV-algebras from $\mathcal M$ is closed under homomorphic images and subalgebras, but it is not closed under direct products.

Moreover, {\rm (1)} if $N$ is a homomorphic image of $M$, then $N\cong \Gamma(H_1\lex G_1,(u_1,0))$, where $(H_1,u_1)$ and $G_1$ are homomorphic images of $(H,u)$ and $G$, respectively.

{\rm (2)} If $N$ is a subalgebra of $M$, then $N\cong \Gamma(H_0\lex G_0,(u_0,0))$, where $(H_0,u)$ and $G_0$ are subalgebras of $(H,u)$ and $G$, respectively.

\end{theorem}

\begin{proof}
Let $I$ be a lexicographic ideal of $M$ such that $M$ is $I$-representable.

(1) Let $f:M\to N$ be a surjective homomorphism. Then $N$ is symmetric and from $\mathcal M$ whilst $\mathcal M$ is a variety. If we set $f(I)=\{f(x): x \in I\}$, then $f(I)$ is a normal ideal of $N=f(M)$ that is also commutative, prime and strict. We claim that $f(I)$ is a retractive ideal, too. Let $\pi_I: M\to M/I$ be the canonical projection and let $\delta_I:M/I\to M$ be a homomorphism such that $\pi_I \circ \delta_I = id_{M/I}$. Let $M_0=\delta_I(M/I)$ be a subalgebra of $M$ that is isomorphic to $M/I$. If we define $\hat f: M/I \to N/f(I)$  by $\hat f(x/I)= f(x)/f(I)$, then $\hat f$ is a well-defined homomorphism such that $\hat f \circ \pi_I= \pi_{f(I)}\circ f$. Set $N_0=f(M_0)$ and let $f_{M_0}$ be the restriction of $f$ onto $M_0$. We define $\delta_{f(I)}: N/f(I) \to N$ via $\delta_{f(I)}(f(x)/f(I)):= f_{M_0}(\delta_I(x/I))$; then $\delta_{f(I)}$ is a well-defined homomorphism such that $\delta_{f(I)}(N/f(I))=N_0$ and $f_{M_0} \circ \delta_I = \delta_{f(I)} \circ \hat f$. Hence,

\begin{eqnarray*}
\pi_{f(I)}\circ \delta_{f(I)}(f(x)/f(I)) &=& \pi_{f(I)} \circ f_{M_0} \circ \delta_I(x/I)\\
&=& \hat f \circ \pi_I \circ \delta_I(x/I)= \hat f(x/I)\\
&=& f(x)/f(I)
\end{eqnarray*}
that proves $f(I)$ is a retractive ideal of $N$.

Take the unital representation of pseudo MV-algebras given by Theorem \ref{th:2.1}, and let $N\cong \Gamma(K,v)$ and let $f:(H\lex G,(u,0)) \to (K,v)$ be a surjective homomorphism of unital $\ell$-groups.  Let $f_1(h)=f(h,0)$, $h \in H$, and $f_2(g)=f(0,g)$, $g \in G$. If we set $H_1:=f_1(H)$, $u_1=f(u,0)$, and $G_1:=f_2(G)$. Then $N \cong \Gamma(H_1\lex G_1,(u_1,0))$.

(2) Let $N$ be a subalgebra of $M$. Then $N$ is symmetric and belongs to $\mathcal M$. We set $J:=N\cap I$. Then $J$ is a normal ideal of $N$ that is also commutative and prime. It is strict, too, because if $x \in N$ and $x \notin J$, then $x \notin I$ and $x>y$ for any $y \in J$ and consequently, for any $y \in J$. Then $N/J$ can be embedded into $M/I$ by a mapping $i_J(x/J):=x/I$ ($x\in N$) and if $i_0(x)=x$, $x \in N$, then $\pi_I\circ i_0=i_J\circ \pi_J$. Let $M_0:=\delta_I(M/I)$ and $N_0:=M_0\cap N$. Then $\delta_I(N/I)\in N_0$; indeed, if there is $x \in N_0$ such that $\delta_I(x/I)\notin N_0$, then $\pi_I\circ \delta_I(x/I)=x/I \notin N_0/I$. Define $\delta_J: N/J\to N$ by $\delta_J(x/J)=i_J^{-1}\circ\delta_I(x/I)$. Since $i^{-1}_I\circ \pi_I(x)=\pi_J\circ i^{-1}_0(x)$, $x \in N$, then
\begin{eqnarray*}
\pi_J \circ \delta_J(x/J)&=& \pi_J\circ i_0^{-1}\circ \delta_I(x/I)\\
&=& i_I^{-1}\circ\pi_I \circ \delta_I(x/I) = i^{-1}_I(x/I)=x/J.
\end{eqnarray*}

The rest follows the analogous steps as the end of (1).

(3) According to Corollary \ref{co:state2}, every lexicographic pseudo MV-algebra $M$ admits a unique state. But the pseudo MV-algebra $M\times M$ admits two extremal states, and therefore, $M\times M$ is not lexicographic.
\end{proof}

We note that in case (3) of latter Theorem if $I$ is a lexicographic ideal of $M$, then $I\times I$ is by Proposition \ref{pr:product} a retractive ideal but not lexicographic.

\section{Categorical Representation of Strong $(H,u)$-perfect Pseudo MV-algebras}%8

In this section, we establish the categorical equivalence of the category of strong $(H,u)$-perfect pseudo MV-algebras with the variety of $\ell$-groups. This extends the categorical representation of strong $n$-perfect pseudo MV-algebras from \cite{Dv08} and of $\mathbb H$-perfect pseudo MV-algebras from \cite{264} with the variety of $\ell$-groups. In what follows, we follow the ideas of \cite[Sec 5]{264} and to be self-contained we repeat them mutatis mutandis.

Let $\mathcal {SPP}_s\mathcal{MV}_{H,u}$ be the category of strong $(H,u)$-perfect pseudo MV-algebras whose objects are strong $(H,u)$-perfect pseudo MV-algebras and morphisms are homomorphisms of pseudo MV-algebras. Now let $\mathcal G$ be the category whose objects are $\ell$-groups  and morphisms are homomorphisms of $\ell$-groups.

Define a mapping $\mathcal M_{H,u}: \mathcal G \to  \mathcal {SPP}_s\mathcal{MV}_{H,u}$ as follows: for $G\in \mathcal G,$ let
$$
\mathcal M_{H,u}(G):= \Gamma(H\lex G,(u,0))
$$
and if $h: G \to G_1$ is an $\ell$-group homomorphism, then

$$
\mathcal M_{H,u}(h)(t,g)= (t, h(g)), \quad (t,g) \in \Gamma(H\lex G,(u,0)).
$$
It is easy to see that $\mathcal M_{H,u}$ is a functor.

\begin{proposition}\label{pr:4.1}
$\mathcal M_{H,u}$ is a faithful and full
functor from the category ${\mathcal G}$ of $\ell$-groups  into the
category $\mathcal{SPP}_s\mathcal{MV}_{H,u}$ of strong $(H,u)$-perfect pseudo MV-algebras.
\end{proposition}

\begin{proof}
Let $h_1$ and $h_2$ be two morphisms from $G$
into $G'$ such that $\mathcal M_{H,u}(h_1) = \mathcal M_{H,u}(h_2)$. Then
$(0,h_1(g)) = (0,h_2(g))$ for each $g \in G^+$, consequently $h_1 =
h_2.$

To prove that $\mathcal M_{H,u}$ is a full  functor, suppose that $f$ is a morphism from a strong $(H,u)$-perfect pseudo MV-algebra
$\Gamma(H\lex G, (u,0))$ into some $\Gamma(
H\lex G_1, (u,0)).$  Then $f(0,g)
= (0,g')$ for a unique $g' \in G'^+$. Define a mapping $h:\ G^+ \to
G'^+$ by $h(g) = g'$ iff $f(0,g) =(0,g').$ Then $h(g_1+g_2) = h(g_1)
+ h(g_2)$ if $g_1,g_2 \in G^+.$
Assume now that $g \in G$ is arbitrary. Then $g=g^+-g_1,$ where $g^+=g\vee 0$ and $g^-=-(g\wedge 0)$, and $g=-g^-+g^+$. If $g=g_1-g_2,$ where $g_1,g_2 \in G^+$, then $g^++g_2=g^-+g_1$ and $h(g^+)+ h(g_2)=h(g^-)+h(g_1)$ which shows that $h(g) = h(g_1) - h(g_2)$ is a
well-defined extension of $h$ from $G^+$ onto $G$.

Let $0\le g_1 \le g_2.$ Then $(0,g_1)\le (0,g_2),$
which means  $h$ is a mapping preserving the partial order.

We have yet to show that $h$ preserves $\wedge$ in $G$, i.e., $h(a \wedge b) = h(a) \wedge h(b)$ whenever $a,b \in G.$ Let $a=  a^+- a^-$ and $b=
b^+- b^-$, and $a =-a^- +a^+$, $b = -b^- + b^+$. Since , $h((a^+
+b^-) \wedge (a^- + b^+)) = h(a^+ +b^-) \wedge h(a^- + b^+).$
Subtracting $h(b^-)$ from the right hand and $h(a^-)$ from the left
hand, we obtain the statement  in question.

Finally, we have proved that $h$ is a homomorphism of $\ell$-groups, and $\mathcal M_{H,u}(h) = f$ as claimed.
\end{proof}

We note that by a {\it universal group}  for a
pseudo MV-algebra $M$ we mean a pair $(G,\gamma)$ consisting of an
$\ell$-group $G$ and a $G$-valued measure $\gamma :\, M\to G^+$
(i.e., $\gamma (a+b) = \gamma(a) + \gamma(b)$ whenever $a+b$ is
defined in $M$) such that the following conditions hold: {\rm (i)}
$\gamma(M)$ generates ${ G}$. {\rm (ii)} If $K$ is a group and
$\phi:\, M\to K$ is an $K$-valued measure, then there is a group
homomorphism ${\phi}^*:{ G}\to K$ such that $\phi ={\phi}^*\circ
\gamma$.

Due to \cite{151}, every pseudo MV-algebra admits a universal group,
which is unique up to isomorphism, and $\phi^*$ is unique. The
universal group for $M = \Gamma(G,u)$ is $(G,id)$ where $id$ is the
embedding of $M$ into $G$.

Let $\mathcal A$ and $\mathcal B$ be two categories and let $f:\mathcal A \to \mathcal B$ be a functor. Suppose that $g,h$ be two functors from $\mathcal B$ to $\mathcal A$ such that $g\circ f = id_\mathcal A$ and $f\circ h = id_\mathcal B,$ then $g$ is a {\it left-adjoint} of $f$ and $h$ is a {\it right-adjoint} of $f.$

\begin{proposition}\label{pr:4.2}
The functor $\mathcal  M_{H,u}$ from the
category ${\mathcal  G}$ into  $\mathcal{SPP}_s\mathcal{MV}_{H,u}$ has  a left-adjoint.
\end{proposition}

\begin{proof}
We show, for a strong $(H,u)$-perfect  pseudo MV-algebra $M$ with an  $(H,u)$-decomposition $(M_t: t \in [0,u]_H)$ and an $(H,u)$-strong cyclic family $(c_t: t \in [0,u]_H)$  of elements of $M$, there is a universal arrow $(G,f)$, i.e., $G$ is an object in $\mathcal G$ and $f$ is a homomorphism from the pseudo MV-algebra
$M$ into ${\mathcal  M}_{H,u}(G)$ such that if $G'$ is an object from ${\mathcal G}$ and $f'$ is a homomorphism from $M$ into ${\mathcal  M}_{H,u}(G')$, then
there exists a unique morphism $f^*:\, G \to G'$ such that ${\mathcal
M}_{H,u}(f^*)\circ f = f'$.

By Theorem \ref{th:3.5}, there is a unique (up to isomorphism of $\ell$-groups) $\ell$-group $G$  such that $M \cong \Gamma(H \lex G,(u,0)).$ By \cite[Thm 5.3]{151}, $(H \lex G, \gamma)$ is a universal group for $M,$ where $\gamma: M \to  \Gamma(H \lex G, (u,0))$ is defined by $\gamma(a) = (t,a -c_t),$ if $a \in M_t.$
\end{proof}

Define a mapping ${\mathcal  P}_{H,u}: \mathcal  {SPP}_s\mathcal{MV}_{H,u}\to {\mathcal  G}$
via ${\mathcal  P}_{H,u}(M) := G$ whenever $(H\lex  G, f)$ is a
universal group for $M$. It is clear that if $f_0$ is a morphism
from the pseudo MV-algebra $M$ into another one $N$, then $f_0$ can be uniquely extended to an $\ell$-group homomorphism ${\mathcal  P}_{H,u} (f_0)$ from $G$ into $G_1$, where $(H
\lex G_1, f_1)$ is a universal group for the strong $(H,u)$-perfect
pseudo MV-algebra $N$.

\begin{proposition}\label{pr:4.3}
The mapping ${\mathcal  P}_{H,u}$ is a functor from the
category $\mathcal {SPP}_s\mathcal{MV}_{H,u}$ into the category ${\mathcal  G}$ which is a
left-adjoint of the functor ${\mathcal  M}_{H,u}.$
\end{proposition}

\begin{proof}
 It follows from the properties of the
universal group.
\end{proof}

Now we present  the basic result of this section on a categorical equivalence of the
category of strong $(H,u)$-perfect pseudo MV-algebras and the category of $\mathcal G.$

\begin{theorem}\label{th:4.4}
The functor ${\mathcal  M}_{H,u}$ defines a categorical
equivalence of the category ${\mathcal  G}$   and the
category $\mathcal {SPP}_s\mathcal{MV}_{H,u}$ of strong $(H,u)$-perfect pseudo MV-algebras.

In addition, suppose that $h:\ {\mathcal  M}_{H,u}\mathbb (G) \to {\mathcal  M}_{H,u}(G')$ is a
homomorphism of pseudo  MV-algebras, then there is a unique homomorphism
$f:\ G \to G'$ of  $\ell$-groups such that $h = {\mathcal  M}_{H,u}(f)$, and
\begin{enumerate}
\item[{\rm (i)}] if $h$ is surjective, so is $f$;
 \item[{\rm (ii)}] if $h$ is  injective, so is $f$.
\end{enumerate}
\end{theorem}

\begin{proof}
According to \cite[Thm IV.4.1]{MaL}, it is
necessary to show that, for a strong $(H,u)$-perfect pseudo MV-algebra $M$, there is an
object $G$ in ${\mathcal  G}$ such that ${\mathcal  M}_{H,u}(G)$ is isomorphic to $M$. To show that, we take a universal group $(H
\lex G, f)$. Then ${\mathcal  M}_{H,u}(G)$ and $M$ are isomorphic.
\end{proof}

An important kind of $\ell$-groups are doubly transitive $\ell$-groups; for more details on them see e.g. \cite{Gla}. Every such an $\ell$-group generates the variety of $\ell$-groups, \cite[Lem 10.3.1]{Gla}. The notion of doubly transitive unital $\ell$-group $(G,u)$ was introduced and studied in \cite{DvHo}, and according to \cite[Cor 4.9]{DvHo}, the pseudo MV-algebra $\Gamma(G,u)$ generates the variety of pseudo MV-algebras.

An example of a doubly transitive  permutation $\ell$-group is the
system of all automorphisms, $\mbox{Aut}(\mathbb R)$, of the real
line $\mathbb R$, or the next example:

Let $u\in \mbox{Aut}(\mathbb R)$ be the translation $t u=t+1$, $t
\in \mathbb R$, and
$${\rm BAut}(\mathbb R)=\{g\in\mbox{Aut}(\mathbb R):\ \exists\ n\in{\mathbb
N},\,\,\, u^{-n}\le g\le u^n\}.
$$
Then $(\mbox{\rm BAut}(\mathbb R),u)$ is a   doubly transitive
unital $\ell$-permutation group, and  it is a generator of the variety of pseudo MV-algebras $\mathcal{P}_s\mathcal{MV}$. In addition, $\Gamma(\mbox{\rm BAut}(\mathbb R),u)$ is a stateless pseudo MV-algebra.

The proof of the following statement is practically the same as that of \cite[Thm 5.6]{264} and therefore, we omit it here.

\begin{theorem}\label{th:4.6} Let $G$ be a  doubly transitive
$\ell$-group.  Then  the variety generated by $\mathcal{SPP}_s\mathcal{MV}_{H,u}$ coincides with the variety generated by ${\mathcal M}_{H,u}(G)$.
\end{theorem}

\section{Weak $(H,u)$-perfect Pseudo MV-algebras}%9

In this section, we will study another kind of $(H,u)$-perfect pseudo MV-algebras, called weak $(H,u)$-perfect pseudo MV-algebras. Their prototypical examples are pseudo MV-algebras of the form $\Gamma(H\lex G,(1,b))$, where $(H,u)$ is an Abelian unital $\ell$-group, $G$ is an $\ell$-group and $b\in G$. Such pseudo MV-algebras were studied for the case $(H,u)=(\mathbb H,1)$ in \cite{264}.

Let $(H,u)$ be an Abelian unital $\ell$-group.
We say that a pseudo MV-algebra $M\cong \Gamma(K,v)$ with an $(H,u)$-decomposition $(M_t: t \in [0,u]_H)$ is {\it weak} if there is a system $(c_t: t \in [0,u]_ H)$ of elements of $M$ such that (i) $c_0=0,$ (ii) $c_t \in C(K)\cap M_t,$ for any $t \in [0,u]_H,$ and (iii) $c_{v+t}=c_v+c_t$ whenever $v+t \le u.$

We notice that in contrast to the strong cyclic property, we do not assume $c_1=1.$ In addition, a weak $(H,u)$-perfect pseudo MV-algebra $M$ is strong iff $c_1 =1.$

\begin{example}\label{ex:weak}
Let $(H,u)$ be an Abelian unital $\ell$-group. The pseudo MV-algebra $M=\Gamma(H\lex G,(u,b))$, where $b\in G$,  $M_t=\{(t,g): (t,g)\in M\}$, $t \in [0,u]_H$ form an $(H,u)$-decomposition of $M$, is a weak pseudo MV-algebra setting $c_t=(t,0)$, $t \in [0,u]_H$.
\end{example}

\begin{proof}
We have to verify that $(M_t: t \in [0,u]_H)$ is an $(H,u)$-decomposition. To show that it is enough to verify (b) of Definition \ref{de:3.1}, i.e. $M_t^-=M_{u-t}=M_t^\sim$ for each $t\in [0,u]_H$. Let $(t,g)\in M_t$. Then $(t,g)^-=(u,b)-(t,g)=(u-t,b-g).$ If we choose $(t,g_0)$, where $g_0=b+g-b$, then $(t,g_0)^\sim = -(t,g_0)+(u,b)= (-t+u,-g_0+b)=(u-t,b-g)=(t,g)^-$ which yields $(t,g)^-\in M_t^\sim$, that is $M_t^-\subseteq M_t^\sim.$ Dually we show $M_t^\sim\subseteq  M_t^-$. Then $M_t^-=M_{u-t}=M_t^\sim$.
\end{proof}

Whereas every strong $(H,u)$-perfect pseudo MV-algebra is symmetric, weak ones are not necessarily symmetric.

For example, the pseudo MV-algebra $\Gamma(H\lex G,(u,b))$, where $b>0$ and $b\notin C(G)$ and $M_t:=\{(t,g)\in \Gamma(H\lex G,(u,b))\}$ for each $t \in [0,u]_H$,
is weak $(H,u)$-perfect but neither strong $(H,u)$-perfect nor symmetric.

We note that $M_0$ is a unique maximal and normal ideal of $M$. This ideal is retractive iff $M$ is strongly $(H,u)$-perfect. For example, let $M=\Gamma(\mathbb Z \lex \mathbb Z, (2,1))$. Then $M$ is weakly $(\mathbb Z,2)$-perfect that is not strongly $(\mathbb Z,2)$-perfect, and $M_0=\{(0,n):n \ge 0\}$, $M_1=\{(1,n): n\in \mathbb Z\}$, $M_2 =\{(2,n): n \le 1\}$, $M/M_0\cong \Gamma(\frac{1}{2}\mathbb Z,1)$ and it has no isomorphic copy in $M$. In addition, $M_0$ is not retractive.

We notice that even a pseudo MV-algebra of the form $\Gamma(H\lex G,(u,b))$ with $b\ne 0$ can be strongly $(H,u)$-perfect. Indeed, let $M=\Gamma(\mathbb Z \lex \mathbb Z, (2,2))$. This MV-algebra is isomorphic with the MV-algebra $M_1:=\Gamma(\mathbb Z \lex \mathbb Z,(2,0))$. In fact, the mapping $\theta:M_1 \to M$ defined by $\theta(0,n)=(0,n)$, $\theta(1,n)=(1,n+1)$ and $\theta(2,n)=(2,n+2)$ is an isomorphism in question. In addition, $M_0=\{(0,n): n \ge 0\}$ is a retractive ideal and a lexicographic ideal of $M$; $M/M_0=\Gamma(\frac{1}{2}\mathbb Z,1)$ and its isomorphic copy in $M$ is the subalgebra $\{(0,0), (1,1), (2,2)\}$.

The next result is a representation of weak $(H,u)$-perfect pseudo MV-algebras by lexicographic product.

\begin{theorem}\label{th:6.1}
Let $M$ be a weak  $\mathbb H$-perfect pseudo MV-algebra which is not strong. Then there is a
unique (up to isomorphism) $\ell$-group $G$ with an element $b\in
G$, $b\ne 0$, such that $M \cong \Gamma(\mathbb H \lex   G,(1,b)).$
\end{theorem}

\begin{proof}
Assume $M= \Gamma(K,v)$ for some unital $\ell$-group $(H,u)$ is a weak pseudo MV-algebra with a $(H,u)$-decomposition $(M_t: t \in [0,u]_H)$.  Since by (vi) of Theorem \ref{th:3.2} we have $M_0+M_0=M_0$, in the same way as in the proof of Theorem \ref{th:3.5}, there exists an $\ell$-group $G$ such that $G^+=M_0$ and $G$ is a subgroup of $K$.

Since $M$ is not strong,  then $c_1 <1=:u.$ Set $b = 1-c_1 \in M_0\setminus\{0\},$ and define a mapping $\phi:
M\to \Gamma(\mathbb H \lex   G,(1,b))$ as follows
$$
\phi(x) =(t,x-c_t)%\eqno(9.1)
$$
whenever $x \in M_t;$ we note that the subtraction $x-c_t$ is defined in the $\ell$-group $K$.  Using the same way as that in (4.2), we can show that $\phi$ is a well-defined mapping.

We have (1) $\phi(0) =(0,0)$, (2) $\phi(1) = (1,1 -c_1)
= (1,b)$, (3) $\phi(c_t) = (t,0),$ (4)  $\phi(x^\sim) =(1-t, -x +u -c_{1-t})
= (1-t, -x +b +c_t),$  $\phi(x)^\sim =  -\phi(x) +(1,b) = -(t, x-c_t) +(1,b)
= (1-t, -x+b +c_t)$, and similarly (5) $\phi(x^-) = \phi(x)^-.$

Following ideas of the proof of Theorem \ref{th:3.5}, we can prove that $\phi$ is an injective and surjective homomorphism of pseudo MV-algebras as was claimed.
\end{proof}

It is worthy of reminding that Theorem \ref{th:6.1} is a generalization of Theorem \ref{th:3.5}, because Theorem \ref{th:3.5} in fact follows from Theorem \ref{th:6.1} when we have $b=0.$ This happens if $c_1=1$.

Also in an analogous way as in \cite{264}, we establish a categorical equivalence of the category of weak $(H,u)$-perfect pseudo MV-algebras with the category of $\ell$-groups $G$ with a fixed element $b \in G$.

Let  $\mathcal{WPP}_s\mathcal{MV}_{H,u}$ be the category of weak $(H,u) $-perfect pseudo MV-algebras whose objects are weak $(H,u)$-perfect pseudo MV-algebras and morphisms are homomorphisms of pseudo MV-algebras. Similarly, let ${\mathcal L}_{\rm b}$ be the category whose objects are couples $(G,b),$ where $G$ is an $\ell$-group and $b$ is a fixed element from $G$, and morphisms are $\ell$-homomorphisms of $\ell$-groups preserving fixed elements $b$.

Define a mapping $\mathcal F_{H,u}$ from the category $\mathcal L_{\rm b}$ into the category $\mathcal{WPP}_s\mathcal{MV}_{H,u}$ as follows:

Given $(G,b) \in {\mathcal L}_{\rm b},$ we set
$$
{\mathcal F}_{H,u}(G,b) := \Gamma(H\lex  G,(u,b)),
$$
and if $h:(G,b)\to (G_1,b_1),$ then
$$
\mathcal F_{H,u}(h)(t,g)=(t,h(g)),\quad (t,g) \in \Gamma(H\lex G,(u,b)).
$$
It is easy to see that $\mathcal F_{H,u}$ is a functor.

In the same way as the categorical equivalence of strong $(H,u)$-perfect pseudo MV-algebras was proved in the previous section, we can prove the following theorem.

\begin{theorem}\label{th:6.2} The functor ${\mathcal F}_{H,u}$ defines a categorical
equivalence of the category ${\mathcal L}_{\rm b}$ and the category
$\mathcal{WPP}_s\mathcal{MV}_{H,u}$ of weak $(H,u)$-perfect pseudo MV-algebras.
\end{theorem}

Finally, we present addition open problems.

\vspace{3mm}
\begin{problem}
{\rm (1) Find an equational basis for the variety generated by the set $\mathcal {SPP}_s\mathcal{MV}_{H,u}$. For example, if $(H,u)=(\mathbb Z,1)$ the basis is $2.x^2=(2.x)^2$, see \cite[Rem 5.6]{DDT}, and the case $(H,u)=(\mathbb Z,n)$ was described in \cite[Cor 5.8]{Dv08}.}
\end{problem}

%(4)  If $\mbox{LexId}(M)$ is non-empty, does have $\mbox{LexId}(M)$ the greatest ideal? If yes, then it is unique. Is it a kernel of some state, may be extremal ?

(2) Find algebraic conditions that entail that a pseudo MV-algebra is of the form $\Gamma(H\lex G,(u,0)$, where $(H,u)$ is a unital $\ell$-group not necessary Abelian.

\section{Conclusion}%10

In the paper we have established conditions when a pseudo MV-algebra $M$ is an interval in some lexicographic product of an Abelian unital $\ell$-group $(H,u)$ and an $\ell$-group $G$ not necessarily Abelian, i.e. $M =\Gamma(H\lex G,(u,0))$. To show, that we have introduced strong $(H,u)$-perfect pseudo MV-algebras as those pseudo MV-algebras that can be split into comparable slices indexed by the elements from the interval $[0,u]_H$. For them we have established a  representation theorem, Theorem \ref{th:3.5}, and we have shown that the category of strong $(H,u)$-perfect pseudo MV-algebras is categorically equivalent to the variety of $\ell$-groups, Theorem \ref{th:4.4}.

We have shown that our aim can be solved also introducing so-called lexicographic ideals. We establish their properties and Theorem \ref{th:local1} gives also a representation of a pseudo MV-algebra in the form $\Gamma(H\lex G,(u,0))$. We show that every lexicographic pseudo MV-algebra is always local, Theorem \ref{compar}.

Finally, we have studied and represented weak $(H,u)$-perfect pseudo MV-algebras as those that they have a form $\Gamma(H\lex G,(u,g))$ where $g \in G$ is not necessary the zero element, Theorem \ref{th:6.1}.

The present study  has opened a door into a large class of pseudo MV-algebras and formulated new open questions, and we hope that it stimulate a new research on this topic.

\end{document}